\documentclass[12pt]{amsart}
\usepackage[utf8]{inputenc}
\usepackage{color}
\usepackage{float}
\usepackage{times}
\usepackage[hidelinks]{hyperref}
\usepackage{enumerate,latexsym}
\usepackage{latexsym}
\usepackage{subcaption}

\usepackage{fullpage}

\usepackage{amsmath,amsthm,amsfonts,amssymb}
\usepackage{graphicx}
\usepackage{dsfont}

\usepackage{tikz}
\usetikzlibrary{arrows.meta}
\usetikzlibrary{knots}
\usetikzlibrary{hobby}
\usetikzlibrary{arrows,decorations.markings}

\usetikzlibrary{fadings}

\makeatletter
\def\namedlabel#1#2{\begingroup
 #2%
 \def\@currentlabel{#2}%
 \phantomsection\label{#1}\endgroup
}
\makeatother

\usepackage[lite]{amsrefs}

\renewcommand{\PrintDOI}[1]{\href{http://dx.doi.org/\detokenize{#1}}{doi: \detokenize{#1}}%
	\IfEmptyBibField{pages}{, (to appear in print)}{}}

\theoremstyle{plain}
\newtheorem*{theorem*}{Theorem}
\newtheorem*{thmex*}{Theorem~\ref{example}}
\newtheorem*{thmasymp*}{Theorem~\ref{thmAsymp}}
\newtheorem{theorem}{Theorem}[section]

\newtheorem{remark}[theorem]{Remark}

\newtheorem{example}[theorem]{Example}

\theoremstyle{definition}
\newtheorem{definition}[theorem]{Definition}

\newcommand{\Z}{\mathbb{Z}}

\newcommand{\ben}{\begin{enumerate}}
\newcommand{\een}{\end{enumerate}}

\newcommand{\ed}{\end{document}}

\definecolor{rrr}{rgb}{.9,0,.1}

\definecolor{rr}{rgb}{.8,0,.3}

\newcommand{\tr}{\triangleright}

\usepackage{pdfsync}
\graphicspath{ {./images/} }

\title[Oriented Stuck Knots and State Sum Invariant]{RNA foldings, Oriented Stuck Knots and State Sum Invariants}

\author[J. Ceniceros]{Jose Ceniceros}
\address{Hamilton College, Clinton, NY, USA}
\email{jcenicer@hamilton.edu}

\author[M. Elhamadi]{Mohamed Elhamdadi}
\address{University of South Florida, Tampa, Florida, USA}
\email{emohamed@usf.edu}

\author[B. Magill]{Brendan Magill}
\address{Hamilton College, Clinton, NY, USA}
\email{bmagill@hamilton.edu}

\author[G. Rosario]{Gabriana Rosario}
\address{Hamilton College, Clinton, NY, USA}
\email{grosario@hamilton.edu}

\begin{document}

\maketitle

\begin{abstract}
We extend the quandle cocycle invariant to the context of  \emph{stuck links}.  More precisely, we define an invariant of stuck links by assigning Boltzmann weights at both classical and stuck crossings. As an application, we define a single-variable and a two-variable polynomial invariant of stuck links. Furthermore, we define a single-variable and two-variable polynomial invariant of arc diagrams of RNA foldings. We provide explicit computations of the new invariants.
\end{abstract}
\tableofcontents

\section{Introduction}\label{intro}
The RNA molecule is a long chain that consists of a sequence of the base A (adenine), C (cytosine), G (guanine), and U (uracil).  A and U can bond, and C and G can bond with each other.  Thus a folding of RNA corresponds to a word with a matching A-U and C-G for some pairs of letters.  We adopt a mathematical abstraction of an RNA molecule into what is called an \emph{arc diagram} of an RNA folding to study RNA foldings using methods of knot theory.  Here we use the colorings and the state sum invariants of stuck knots and links and apply them to study RNA foldings. State sum invariants of classical knots and links are powerful invariants used to distinguish knots in $3$-space and knotted surfaces in $4$-space; see \cite{CJKLS, CES1, CEGS}.  They use low dimensional cocycles of quandle cohomology as Boltzmann weights at crossings.  One can think of state sum invariants as enhancements of the coloring invariants.  For example, state sum invariants were used in \cite{CEGS} to prove the non-invertibility of a large family of knotted surfaces.  They were also used to determine the minimal triple point number of knotted surfaces \cite{Satoh-Shima}.  State sum invariants were extended to other contexts, such as singular knots in \cite{CCEH, CEM}.

In this article, we extend the quandle cocycle invariant to the context of  \emph{stuck links}.  More precisely, we define an invariant of stuck links by assigning Boltzmann weights at both classical and stuck crossings. As an application, we define a single-variable, two-variable, and three-variable polynomial invariant of stuck links. Furthermore, we define these polynomial invariants for arc diagrams of RNA foldings. Lastly, we provide explicit computations of the new invariants.

The article is organized as follows. Section~\ref{RSK} reviews the basics of stuck knots and links.  Section~\ref{quandles} reviews the algebraic structures needed in this article, namely the notions of quandles, singquandles, and stuquandles are reviewed, and the diagrams leading to their definitions are given. Some examples are also provided.  In Section~\ref{Colorings}, the notion of colorings of stuck links by stuquandles is presented as well as the notion of  \emph{the fundamental stuquandle} associated to a stuck link. Section~\ref{arcdiagrams} uses stuck knots and links to classify RNA foldings through the self-closure of arc diagrams.  In Section~\ref{Weight}, the cocycle invariant is extended to stuck links, and explicit examples are provided.  Section~\ref{CompStuck} gives a few computational examples of the state sum invariant.  The examples were facilitated by a \texttt{Python} and verified by hand computations.  In Section~\ref{CompArcDiagram}, we compute both the counting and state sum invariants of arc diagrams of RNA foldings, thus showing that the state sum invariant is a powerful enhancement of the coloring invariant.

\section{Review of Stuck Knots and Links}\label{RSK}
Stuck links can be considered a generalization of singular links. They were introduced in \cite{B}. Stuck knots and links have applications for modeling biomolecules.  The
relationship between stuck links and RNA folding was given in \cite{CEKL}. This article will follow the definitions and conventions used in \cite{B}. A diagram of a stuck link may contain classical crossings and stuck crossing. A stuck crossing is a singular crossing with additional information. Figure \ref{SX} depicts a  singular crossing, and Figure \ref{StuckX} depicts the two possible stuck crossings. 

\begin{figure}[h!]
    \centering
    \includegraphics[scale=.25]{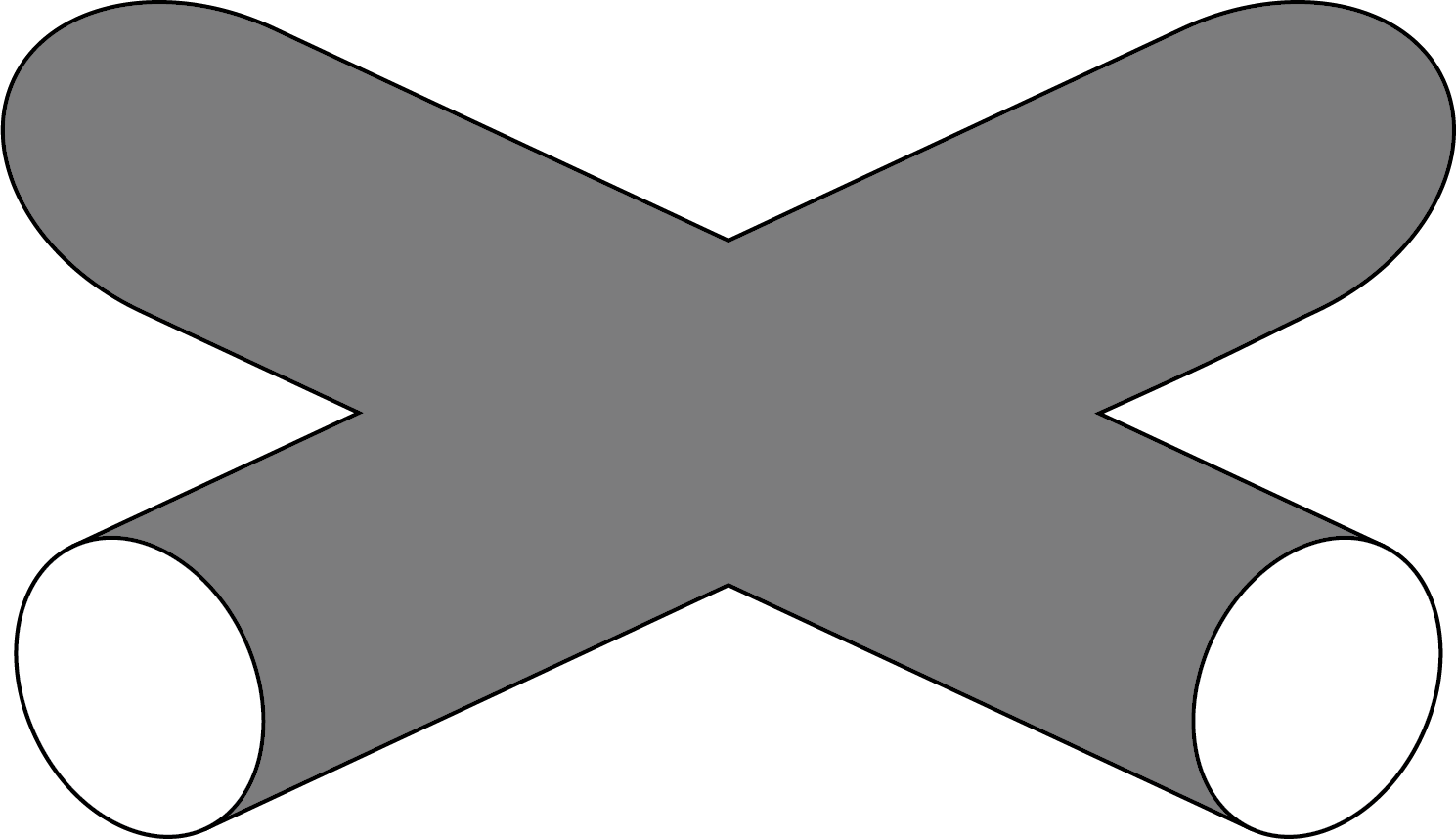}\hspace{2cm}
    \includegraphics[scale=1]{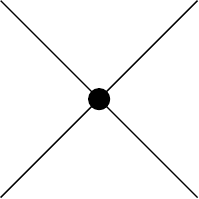}
    \caption{Singular crossing in a singular link (left) and a singular crossing in a singular link diagram (right).}
    \label{SX}
\end{figure}

\begin{figure}[h!]
    \centering
    \includegraphics[scale=.25]{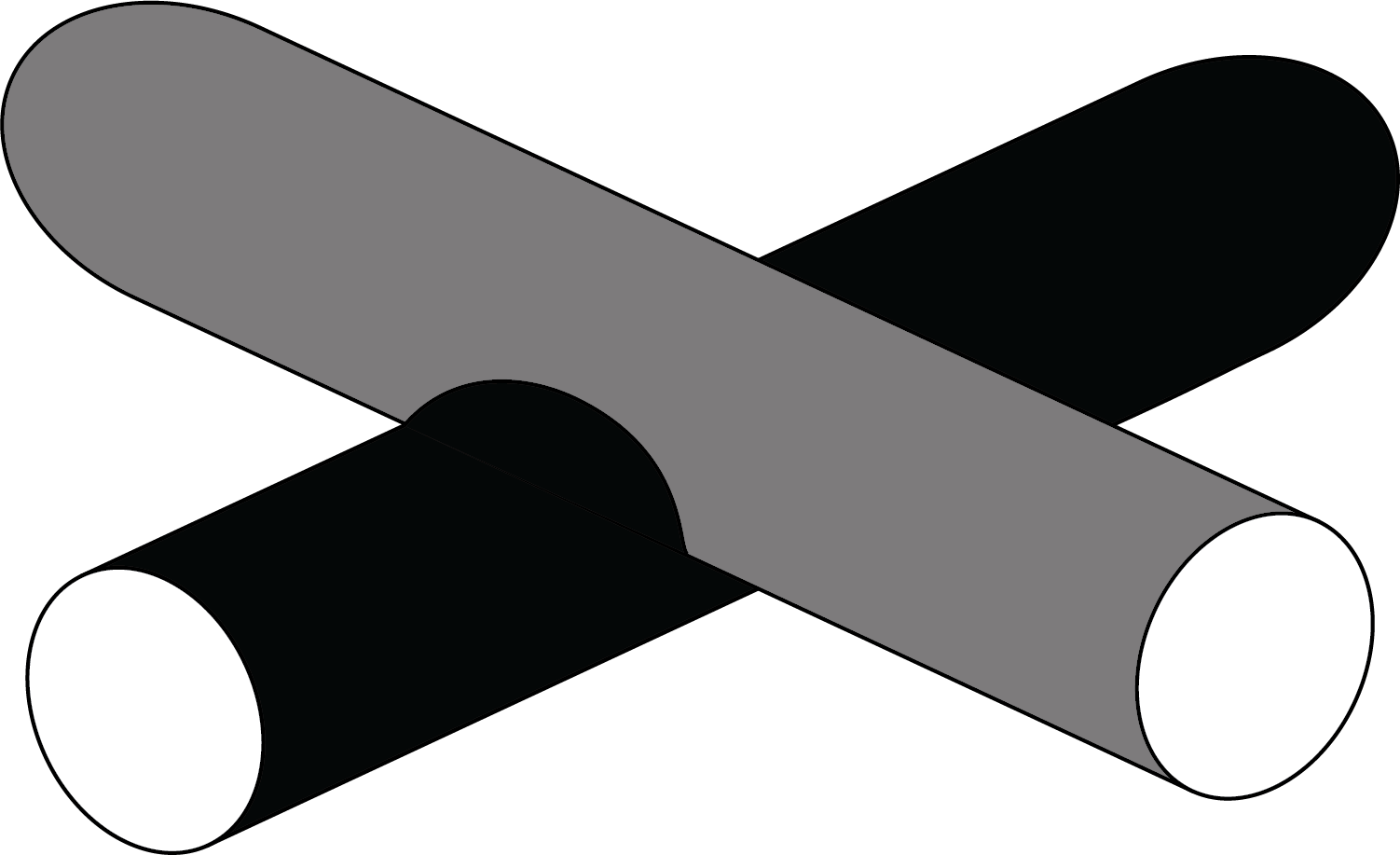}\hspace{2cm}
    \includegraphics[scale=1]{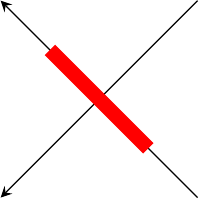}

\vspace{.5cm}
    \includegraphics[scale=.25]{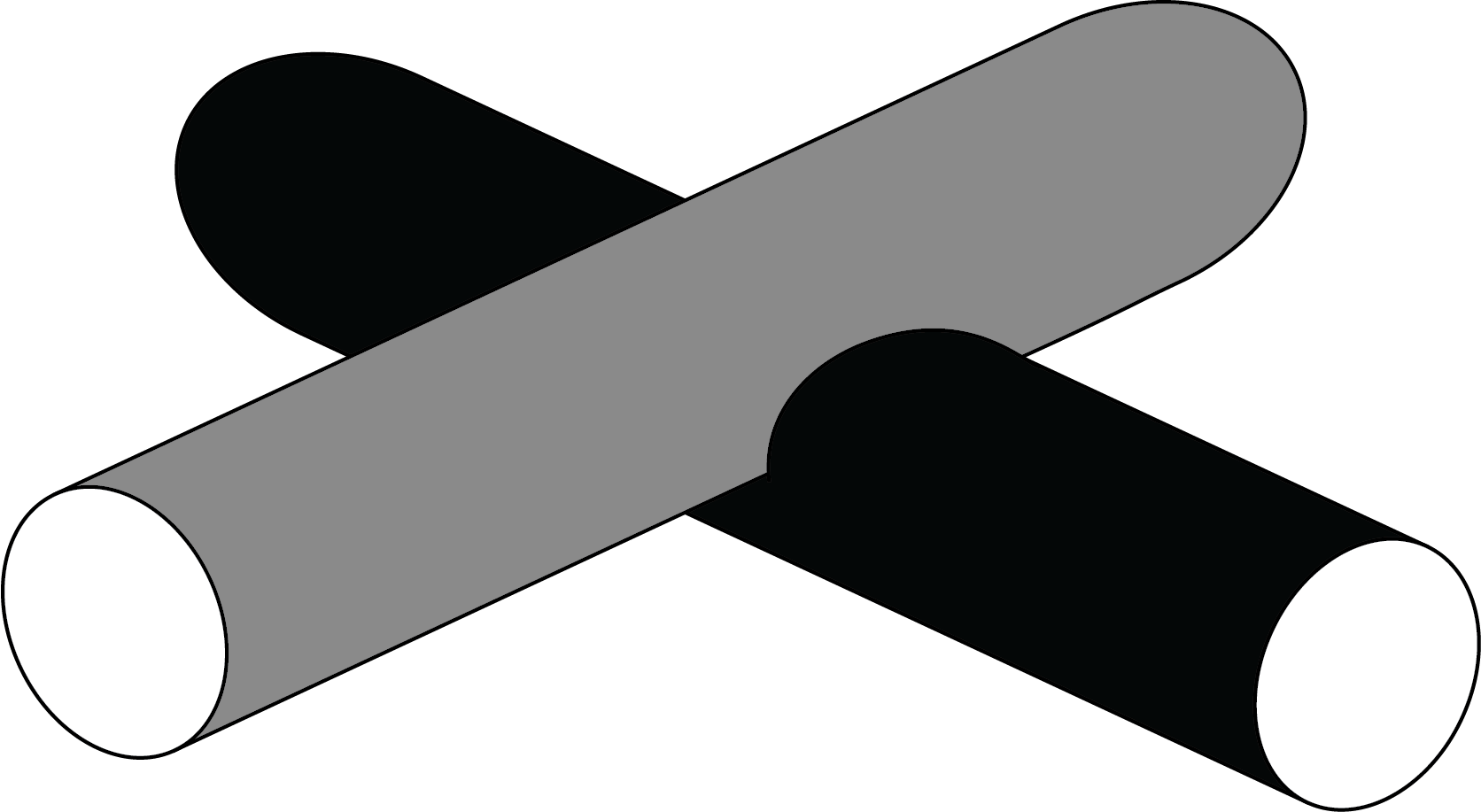}\hspace{2cm}
    \includegraphics[scale=1]{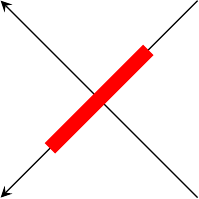}
\caption{Stuck crossings in a stuck link (left) and stuck crossings in a stuck link diagram (right).}
    \label{StuckX}
\end{figure}
In a stuck link diagram, we will use a thick bar on the over arc at a stuck crossing; see Figure \ref{StuckX}. Following the right-hand rule, we will refer to the top crossings in Figure \ref{StuckX} as positive stuck crossings, while the bottom crossings will be negative stuck crossings. 

In \cite{CEKL}, the set of moves in Figure~\ref{stuckrmoves} was shown to be a generating set of oriented stuck Reidemeister moves. The naming convention for the generalized Reidemeister moves in \cite{BEHY} has been adopted and modified to name the oriented stuck Reidemeister moves.
\begin{figure}[h!]
    \centering
\includegraphics[scale=1]{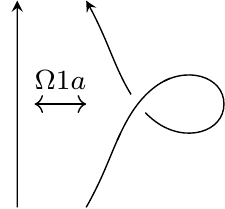} \hspace{2cm}
\includegraphics{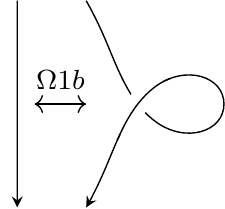}

\vspace{.5cm}
\includegraphics{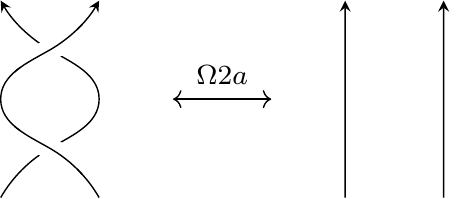}

\vspace{.5cm}
\includegraphics{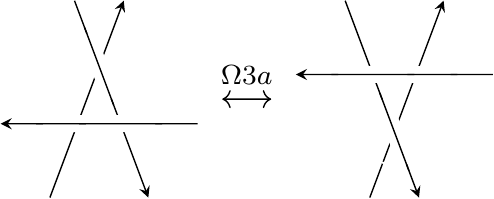}

\vspace{.5cm}
\includegraphics{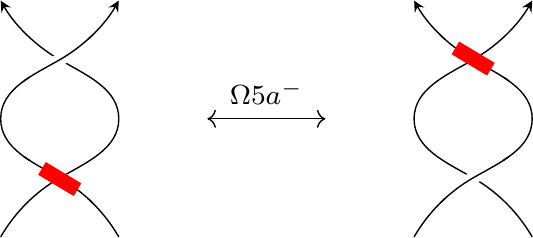}\hspace{2cm}
\includegraphics{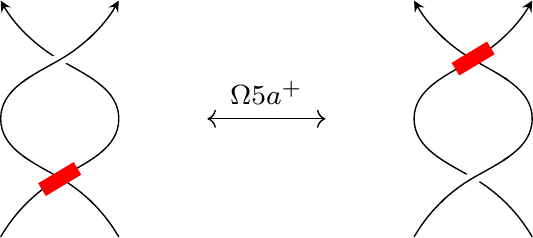}

\vspace{.5cm}
\includegraphics{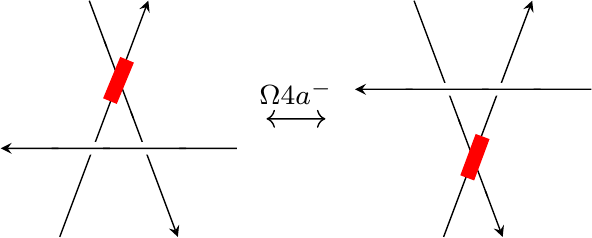}\hspace{2cm}
\includegraphics{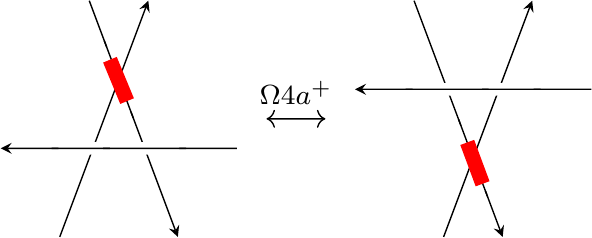}

\vspace{.5cm}
\includegraphics{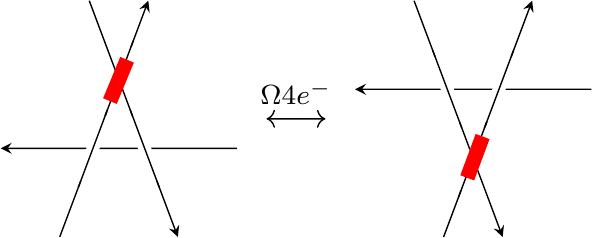}\hspace{2cm}
\includegraphics{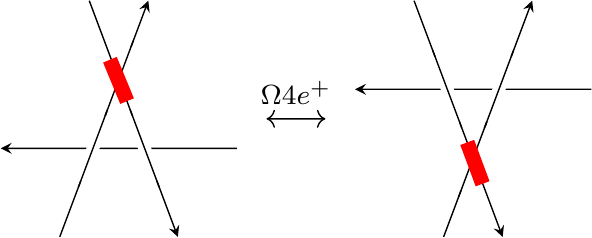}
    \caption{A generating set of oriented stuck Reidemeister moves.}
    \label{stuckrmoves}
\end{figure}

The stuck Reidemeister moves are helpful in the study of stuck links. Specifically, two stuck links diagrams are considered equivalent if one can be obtained from the other by a finite sequence of planar isotopies and oriented stuck Reidemeister moves. Therefore, similar to classical knots and links, a stuck link is an equivalence class of stuck link diagrams.

\section{Algebraic Structure for Stuck Knots} \label{quandles}
In this section, we will review the basics of several algebraic structures that are useful in the study of stuck links. These algebraic structures can be considered the axiomatization of the Reidemeister moves for classical, singular, and stuck knots, respectively.
We will first review quandles and singquandles; more details on quandles can be found in  \cites{EN, Joyce, Matveev} and singquandles in \cites{BEHY, CCE1, CCE2, CCEH}. 

\begin{definition}\label{quandle}  
A set $(X,\ast)$ is called a \emph{quandle} if the following three identities are satisfied.
\begin{eqnarray*}
& &\mbox{\rm (i) \ }   \mbox{\rm  For all $x \in X$,
$x* x =x$.} \label{axiom1} \\
& & \mbox{\rm (ii) \ }\mbox{\rm For all $y,z \in X$, there is a unique $x \in X$ such that 
$ x*y=z$.} \label{axiom2} \\
& &\mbox{\rm (iii) \ }  
\mbox{\rm For all $x,y,z \in X$, we have
$ (x*y)*z=(x*z)*(y*z). $} \label{axiom3} 
\end{eqnarray*}

Furthermore, a quandle is called a \emph{singquandle}, if there exists maps $R_1, R_2 : X\times X \rightarrow X$ satisfying the following conditions,
	\begin{eqnarray}
		R_1(x\,\bar{*}\,y,z)*y&=&R_1(x,z*y), \\
		R_2(x\,\bar{*}\,y, z) & =&  R_2(x,z*y)\,\bar{*}\,y, \\
	      (y\,\bar{*}\,R_1(x,z))*x   &=& (y*R_2(x,z))\,\bar{*}\,z, \\
R_2(x,y)&=&R_1(y,x*y),\\
R_1(x,y)*R_2(x,y)&=&R_2(y,x*y).  	
\end{eqnarray}	
\end{definition} 
\begin{remark}
From Axiom (ii) of Definition~\ref{quandle} we can write the element $x$ as $z \,\bar{*}\, y = x$. Notice that this operation $\bar{*}$ defines a quandle structure on $X$.    
\end{remark}
The following are two typical examples of quandles.
 \begin{enumerate}
     \item Alexander quandles given by $x*y=tx+(1-t)y$ on any $\Lambda=\mathbb{Z}[t,t^{-1}]$-module.\\
     \item Conjugation quandles defined on a group $G$ with $x*y=yxy^{-1}$.
 \end{enumerate}  
The following two examples were given in \cite{BEHY}.
\begin{example}\label{generalizedaffinesingquandle}
Let $a \in \mathbb{Z}_n$ be any invertible element. The pair $(\mathbb{Z}_n, *)$ with operation $x*y = ax+(1-a)y$ is a quandle.  Furthermore, if we have $R_1(x,y) = bx + (1 - b)y$ and $R_2(x,y) = a(1 - b)x + (1 - a(1 - b))y$ where $b$ is any element of $\mathbb{Z}_n$, then $(X,*,R_1,R_2)$  is an oriented singquandle.
\end{example}
The following example gives many singquandle structures defined on groups.
\begin{example}
	Consider the conjugation quandle, $x*y=y^{-1}xy$, on a non-abelian group  $X=G$.  Then $(X, *, R_1, R_2)$ is a singquandle, where for all $x, y \in G$ the maps $R_1$ and $R_2$ are given by:
	\begin{enumerate}
    	\item  $R_1(x,y)=x$ and $R_2(x, y)=y$.
        \item  $R_1(x,y)=xyxy^{-1}x^{-1}$ and $R_2(x, y)=xyx^{-1}$.
        \item  $R_1(x,y)=y^{-1}xy$ and $R_2(x, y)=y^{-1}x^{-1}yxy$.
        \item  $R_1(x,y)=xy^{-1}x^{-1}yx,$ and $R_2(x,y)=x^{-1}y^{-1}xy^2$.
        \item  $R_1(x,y)=y(x^{-1}y)^n$ and $R_2(x, y)=(y^{-1}x)^{n+1}y$, where $n \geq 1$.
    \end{enumerate}
\end{example}   

In this article, we will focus on stuck links and stuquandle structures.  
Now we recall the definition of stuquandle from \cite{CEKL}.

\begin{definition}\label{stuquandle}
	Let $(X, *,R_1,R_2)$ be a singquandle.  Let $R_3$ and $R_4$ be maps from $X \times X$ to $X$.  If $R_3$ and $R_4$ satisfy the following axioms for all $x,y,z \in X$:
	\begin{eqnarray}
		R_3(y,x)*R_4(y,x)&=&R_4(x*y,y),\label{eq6}\\
	    R_4(y,x)&=&R_3(x*y,y),\label{eq7}\\
        R_3(y*x,z)& =&R_3(y,z\,\bar{*}\,x)*x   ,\label{eq8}\\
		R_4(y,z\,\bar{*}\,x)&=&R_4(y*x,z)\,\bar{*}\,x, \label{eq9} \\
	    (x*R_4(y,z))\,\bar{*}\,y &=& (x\,\bar{*}\,R_3(y,z))*z,\label{eq10}
    \end{eqnarray}	
    then the 6-tuple $(X, *, R_1, R_2, R_3, R_4)$ is called an \emph{oriented stuquandle}.
\end{definition}
\noindent Since we will only consider oriented stuquandles, we will refer to them as stuquandles. The following two examples were given in \cite{CEKL}.

\begin{example}\label{generalizedaffinestuquandle}
Let $X=\mathbb{Z}_n$ and
\begin{eqnarray*}
x*y &=& ax+(1-a)y, \\
x \,\bar{\ast}\, y &=& a^{-1}x+(1-a^{-1})y,\\
R_1(x,y) &=& bx + (1 - b)y,\\
R_2(x,y) &=& a(1 - b)x + (1 - a(1 - b))y, \\
R_3(x,y)&=&(1-e)x+ey, \\
R_4(x,y)&=&(1-a(1-e))x+a(1-e)y,
\end{eqnarray*}
where $a$ is any invertible element of $\mathbb{Z}_n$ and any $b,e \in \mathbb{Z}_n$. Then $(X,*,R_1,R_2,R_3,R_4)$ is a stuquandle.   
\end{example}
One can see that the following is a stuquandle as a consequence of Example \ref{generalizedaffinestuquandle}. 

\begin{example}
Let $\Lambda=\mathbb{Z}[t^{\pm 1},v]$  and let $X$ be a $\Lambda$-module. Let
\begin{eqnarray*}
x\ast y &=& tx+(1-t)y,\\
R_1(x,y) &=& \alpha(a,b,c)x+(1-\alpha(a,b,c))y,\\
R_2(x,y) &=& t[1  - \alpha(a,b,c)] x + [1 -t(1-  \alpha(a,b,c))] y,\\
R_3(x,y) &=& (1-\alpha(d,e,f))x+\alpha(d,e,f)y,\\
R_4(x,y) &=&  [1 -t(1-  \alpha(d,e,f))] x + t[1  - \alpha(d,e,f)] y ,\\
\end{eqnarray*}
where $\alpha(a,b,c)=at+bv+ctv$ and $\alpha(d,e,f)=dt+fv+etv$. Then $(X,*,R_1,R_2,R_3,R_4)$ is an oriented stuquandle, which we call an \textit{Alexander oriented stuquandle}. 
\end{example}

\section{Colorings of Stuck Knots by Stuquandles}\label{Colorings}
The idea of coloring a stuck link by a stuquandle was introduced in \cite{CEKL}.  In this section, we recall the notion of the \emph{fundamental stuquandle} associated to a stuck link and use it to define the colorings of stuck links by stuquandles.  Additionally, we provide explicit computations at the end of this section.

\begin{definition}\label{colorrule}
    Let $(X, *, R_1,R_2,R_3,R_4)$ be a stuquandle and $L$ an oriented stuck link diagram. Then a coloring of $L$ by $X$ is an assignment of elements of $X$ to the semiarcs of $L$ at stuck crossings and to the arcs of $L$ at classical crossings obeying the coloring rules in Figure~\ref{newrule}.

\begin{figure}[H]
\centering
\vspace{.5cm}
\includegraphics{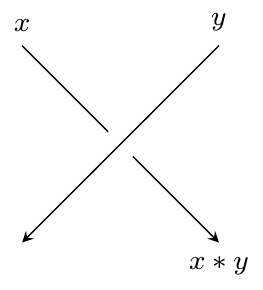}\hspace{2cm}
\includegraphics{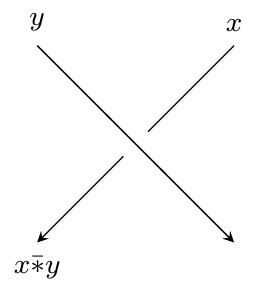}

\vspace{.5cm}
\includegraphics{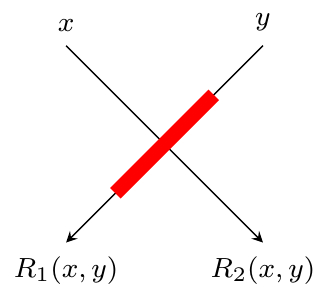}\hspace{2cm}
\includegraphics{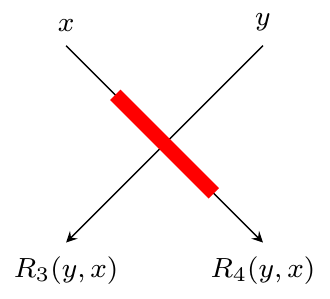}
\caption{Coloring rules at classical and stuck crossings.}
\label{newrule}
\end{figure}
\end{definition}

The stuquandle axioms in Definition \ref{stuquandle} are motivated by the oriented stuck Reidemeister moves for stuck knot diagrams using the coloring rules presented in Figure~\ref{newrule}.

Given a stuck link $L$.  The notion of its fundamental stuquandle $(\mathcal{STQ}(L), \ast, R_1, R_2,R_3, R_4)$ was introduced in \cite{CEKL}.  For more details, the reader is advised to consult \cite{CEKL}.
Furthermore, given a finite stuquandle $(X, \tr, R'_1,R'_2, R'_3, R'_4)$, the set of stuquandle homomorphisms from $\mathcal{STQ(L)}$ to $X$, denoted by $\textup{Hom}(\mathcal{STQ}(L), X)$, was used to define a computable invariant. Specifically, by computing the cardinality $\vert\textup{Hom}(\mathcal{STQ}(L), X) \vert$, we obtain an integer value invariant called the \emph{stuquandle counting invariant}, denoted by $Col_X(L)$. Note that each $f \in \textup{Hom}(\mathcal{STQ}(L), X)$ can be thought of as assigning an element of $X$ to each arc in $L$ at a classical crossing and to each semiarc in $L$ at each stuck crossing satisfying the coloring rules in Figure~\ref{newrule}. The following is an example of a coloring of a stuck knot by a stuquandle.

\begin{example}
Let $X = \Z_6$ be the stuquandle with operations $*,R_1,R_2,R_3,R_4: X\times X \to X$ defined by $x*y = 5x+2y=$ $x\,\bar{*}\,y$ and maps $R_1(x,y) = 4x+3y$, $R_2(x,y) = 3x+4y$, $R_3(x,y) = y$, and $R_4(x,y) = x$. By Example \ref{generalizedaffinestuquandle} with $a = 5, b= 4,$ and $e = 1$, we obtain that $(X,*,R_1,R_2,R_3,R_4)$ is an stuquandle. 

\begin{center}
\includegraphics{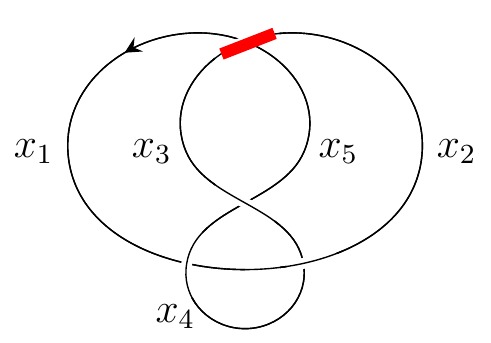}
\end{center}
From the stuck knot diagram, we obtain the following coloring equations,
\begin{align*}
    x_1&= R_3(x_5,x_2),   \\
    x_2&=x_1\,\bar{*}\,x_4,  \\
    x_3&= R_4(x_5,x_2),  \\
    x_4&= x_3\,\bar{*}\,x_2,  \\
    x_5&= x_4*x_3.
\end{align*}
Using the stuquandle operations, we obtain the following equations,
\begin{align*}
    x_1&=  x_2 ,\\
    x_2&=  5x_1+2x_4 ,\\
    x_3&=  x_5, \\
    x_4&=  5x_3+2x_2, \\
    x_5&= 5x_4+2x_3.
\end{align*}
 This system of equations simplifies to give the $12$ elements solution set $\{(x,x,4x,4x,4x), (x,x,4x+3, 4x+3, 4x+3), \; x\in \mathbb{Z}_6\}$.  Each element of the set corresponds to a coloring of the diagram by our stuquandle. 
 Each coloring can be identified with a $5$-tuple $(f(x_1), f(x_2), f(x_3), f(x_4), f(x_5))$, where $f \in \textup{Hom}(\mathcal{STQ}(L),X)$. Therefore, for this stuck Figure-8 knot $3^{k-}_1$, 
\[Col_X(3^{k-}_1)=\vert \textup{Hom}(\mathcal{STQ}(3^{k-}_1), X)\vert = 12.\]
Thus, we obtain the following explicit colorings:
\begin{align*}
\textup{Hom}(\mathcal{STQ}(3^{k-}_1), X) =  \{&(0,0,0,0,0),(3,3,0,0,0),(1,1,1,1,1),(4,4,1,1,1),\\ &(2,2,2,2,2), (5,5,2,2,2), (0,0,3,3,3), (3,3,3,3,3),\\ &(1,1,4,4,4), (4,4,4,4,4), (2,2,5,5,5), (5,5,5,5,5)\}.
\end{align*}
\end{example}
In Section~\ref{Weight}, we will introduce a method to strengthen the stuquandle counting invariant by extracting further information from $\textup{Hom}(\mathcal{STQ}(L), X)$. 

\section{Classifying RNA foldings Using Stuck Knots and Links} \label{arcdiagrams}

Arc diagrams have been used to study RNA foldings; see \cite{KM,TLKL}. As described in \cite{KM}, an RNA molecule is a chain consisting of the bases adenine (A), cytosine (C), uracil (U), and guanine (G). Furthermore, (A) can form a bond with (U) and (C) can form a bond with (G). Therefore, an RNA molecule is a linear sequence of the letters (A), (C), (G), and (U), and an RNA folding is a possible pairing of a given sequence of bases. An RNA folding can be further abstracted by an arc diagram. The following example from \cite{KM} provides an RNA folding and the corresponding arc diagram. 

\begin{example}
Consider the chain $\cdots CCCAAAACCCCCUUUUCCC\cdots$ which has a corresponding folding given in Figure~\ref{RNAchain}.

\begin{figure}[H]
    \centering
    \includegraphics[scale=.5]{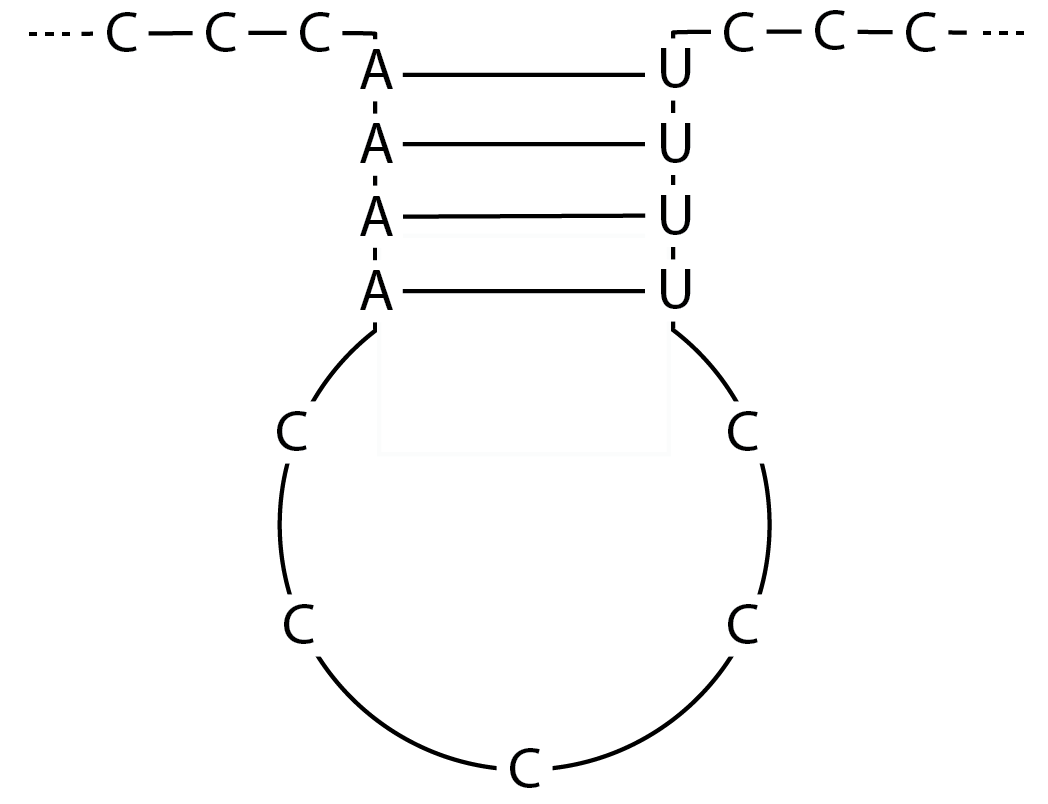}
    \hspace{.5cm}
    \includegraphics[scale=.43]{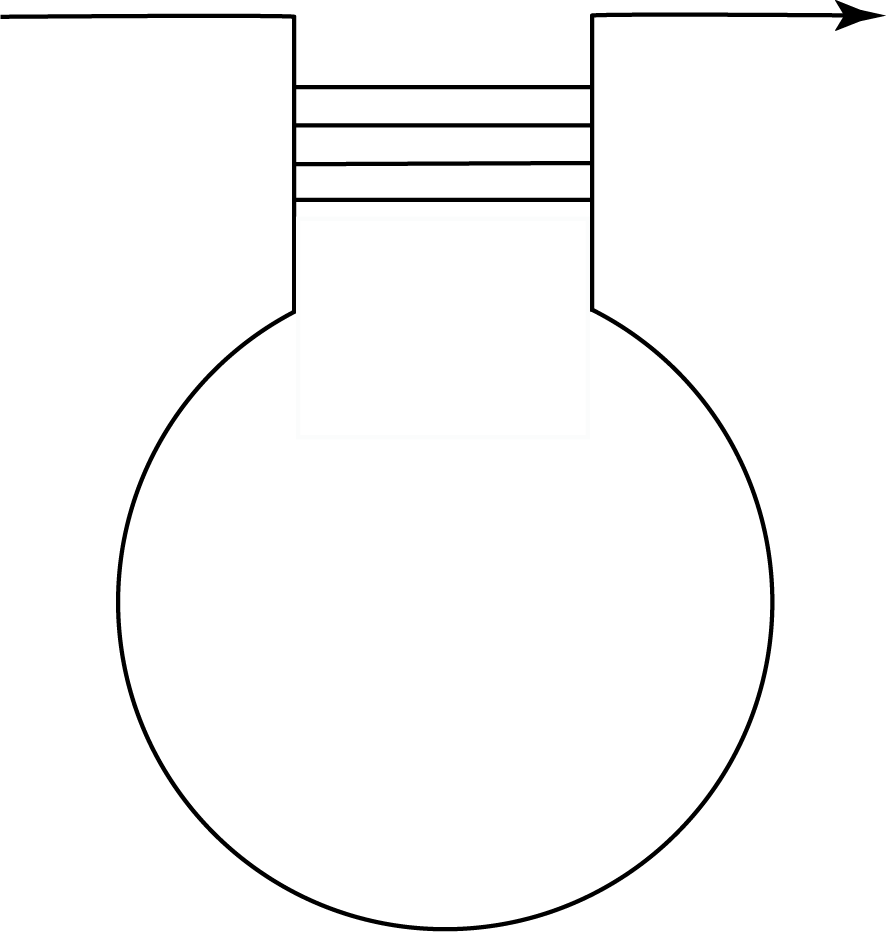}
    \caption{RNA folding and corresponding arc diagram.}
    \label{RNAchain}
\end{figure}

We will use the convention introduced in \cite{B} of representing the bonds with a solid gray rectangle, see Figure~\ref{grayarcdiagram}.
\begin{figure}[H]
    \centering
    \includegraphics[scale=.5]{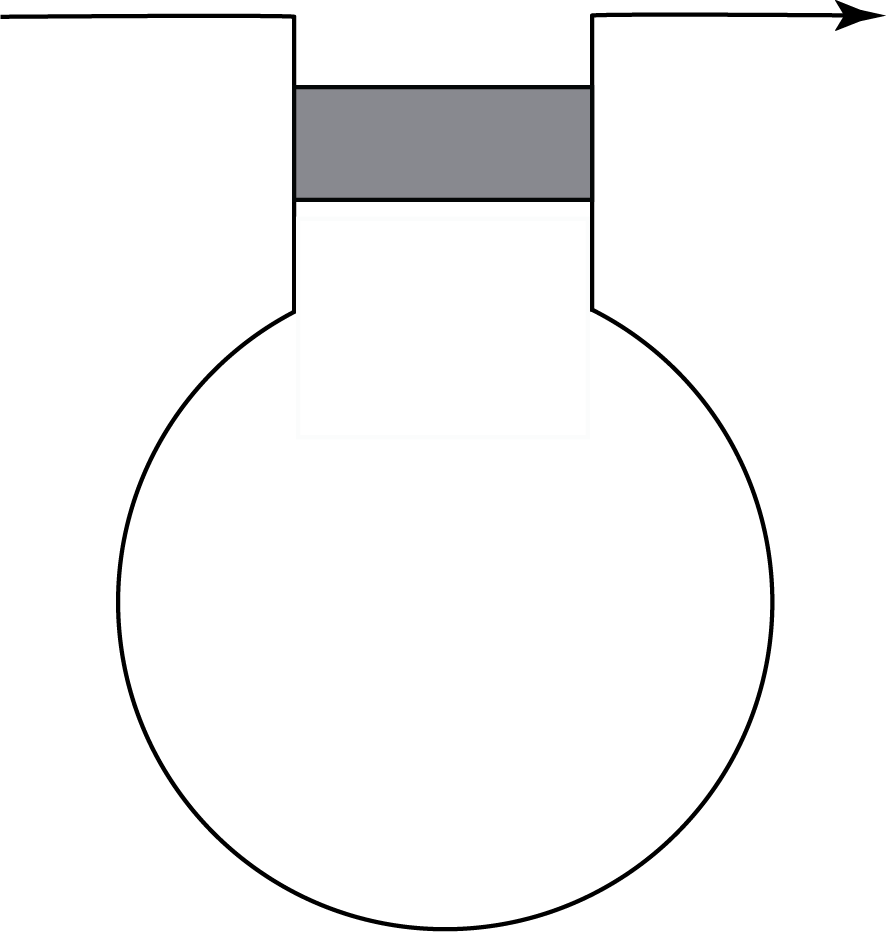}
    \caption{Modified arc diagram.}
    \label{grayarcdiagram}
\end{figure}

\end{example}

Stuck links diagrams provide an alternate way to study the topology of RNA foldings. Specifically, in \cite{B}, a transformation was defined to obtain a stuck link diagram from an arc diagram and vice versa; see Figure~\ref{transformation}. 

\begin{figure}[H]
    \centering
    \includegraphics[scale=.5]{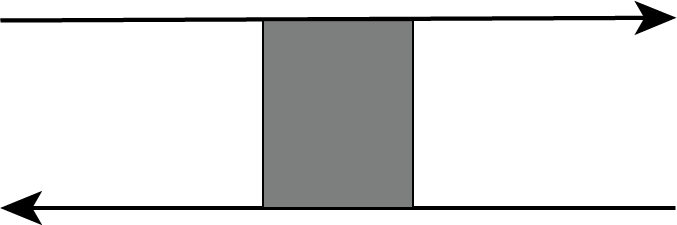}
    \hspace{.6cm}
    \includegraphics{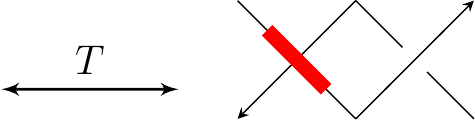}
    \hspace{.6cm}
    \includegraphics{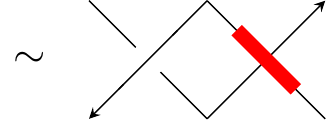}

    \vspace{1.5cm}
    \includegraphics[scale=.5]{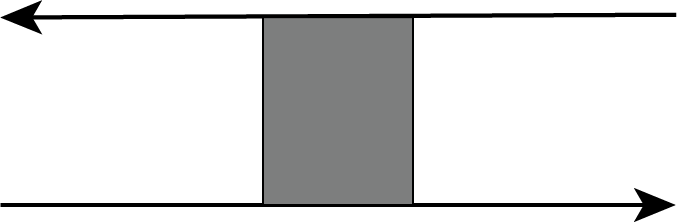}
    \hspace{.6cm}
    \includegraphics{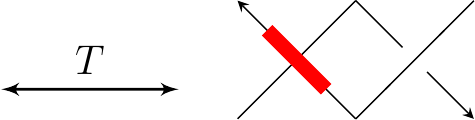}
    \hspace{.6cm}
    \includegraphics{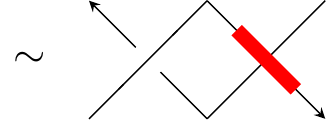}       
\vspace{2cm}

    \includegraphics{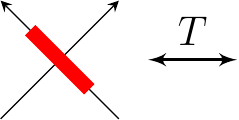}
\hspace{.2cm}
    \includegraphics[scale=.5]{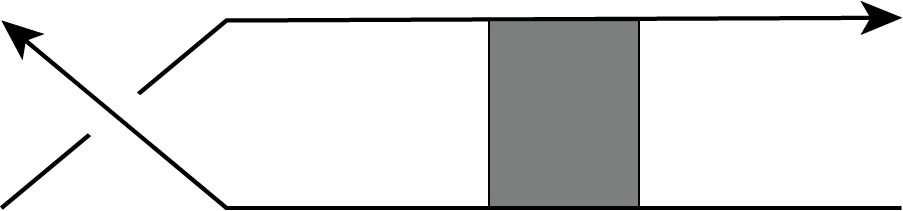}
    \put(10,10){$\sim$}
    \includegraphics[scale=.5]{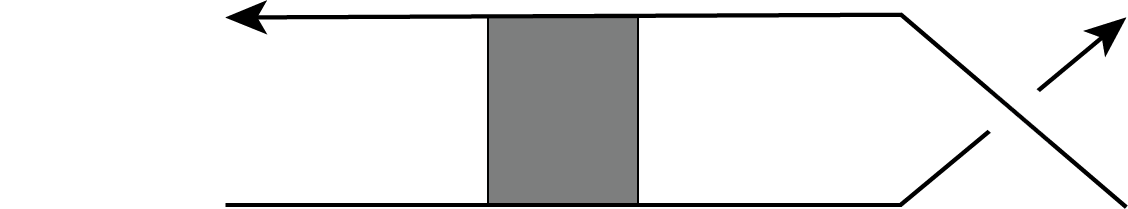}
    \vspace{1.5cm}
    
\includegraphics{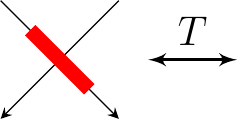}
\hspace{.2cm}    
    \includegraphics[scale=.5]{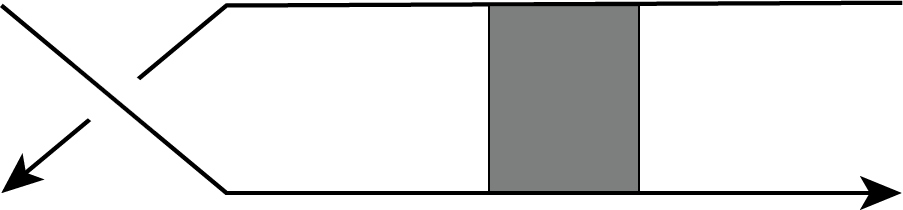}
        \put(10,10){$\sim$}
    \includegraphics[scale=.5]{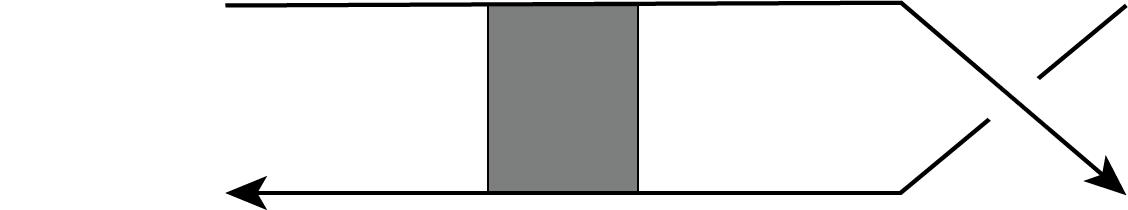}
    \caption{Transformation between arc diagram of RNA folding and stuck diagrams}
    \label{transformation} 
\end{figure}

Additionally, an effective and computable invariant of RNA folding called the \emph{stuquandle counting invariant} was defined in \cite{CEKL}. Before we provide an example of the stuquandle counting invariant of an RNA folding, we note that we will follow the convention set up in \cite{B} of the self-closure of arc diagram to obtain an associated stuck link. By the self-closure, we mean that each strand in an arc diagram connects to itself; see Figure~\ref{Closure}.

\begin{figure}[H]
    \centering
    \includegraphics[scale=.5]{Figure23}
    \hspace{1.5cm}
    \includegraphics[scale=.5]{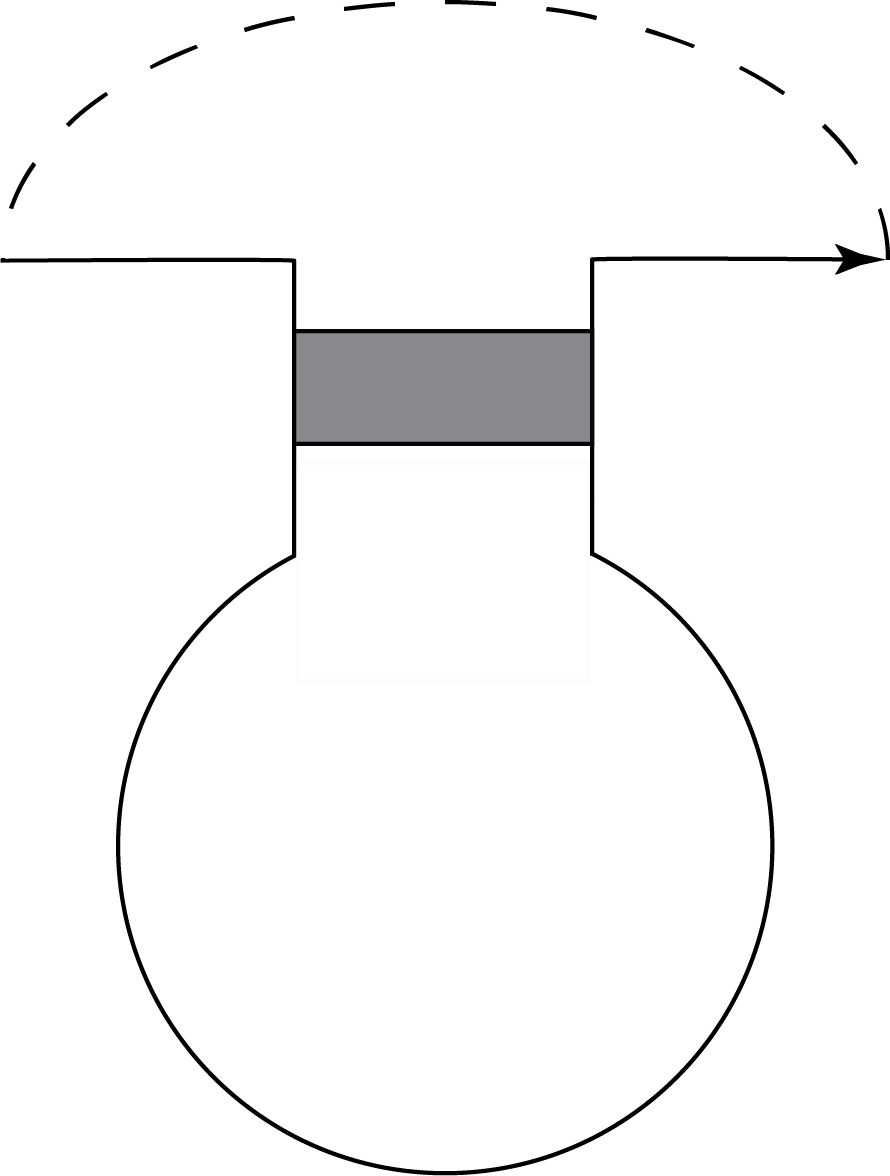}
    \put(-160,50){self-closure}
    \includegraphics{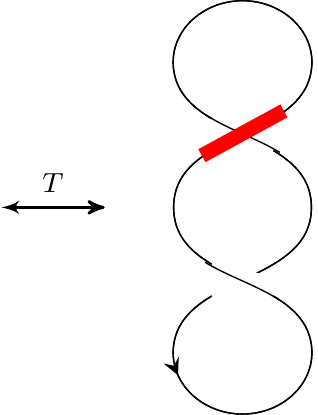}
\vspace{1cm}

    \includegraphics[scale=.5]{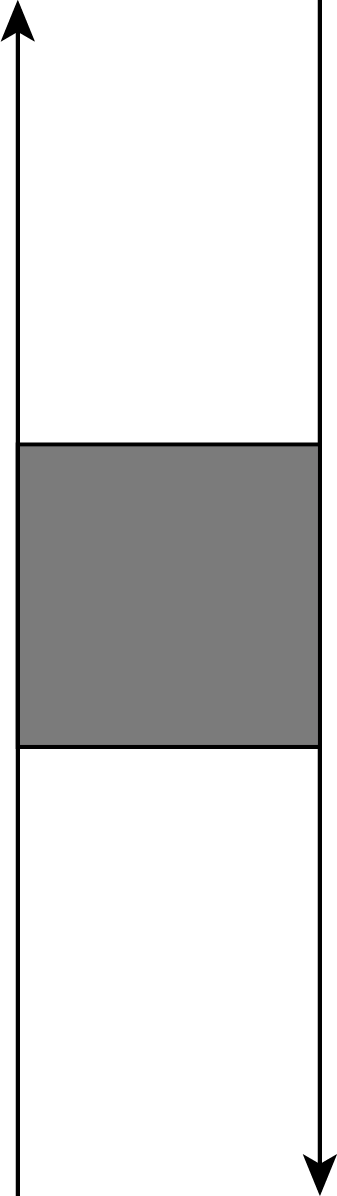}
    \hspace{2.5cm}
    \includegraphics[scale=.5]{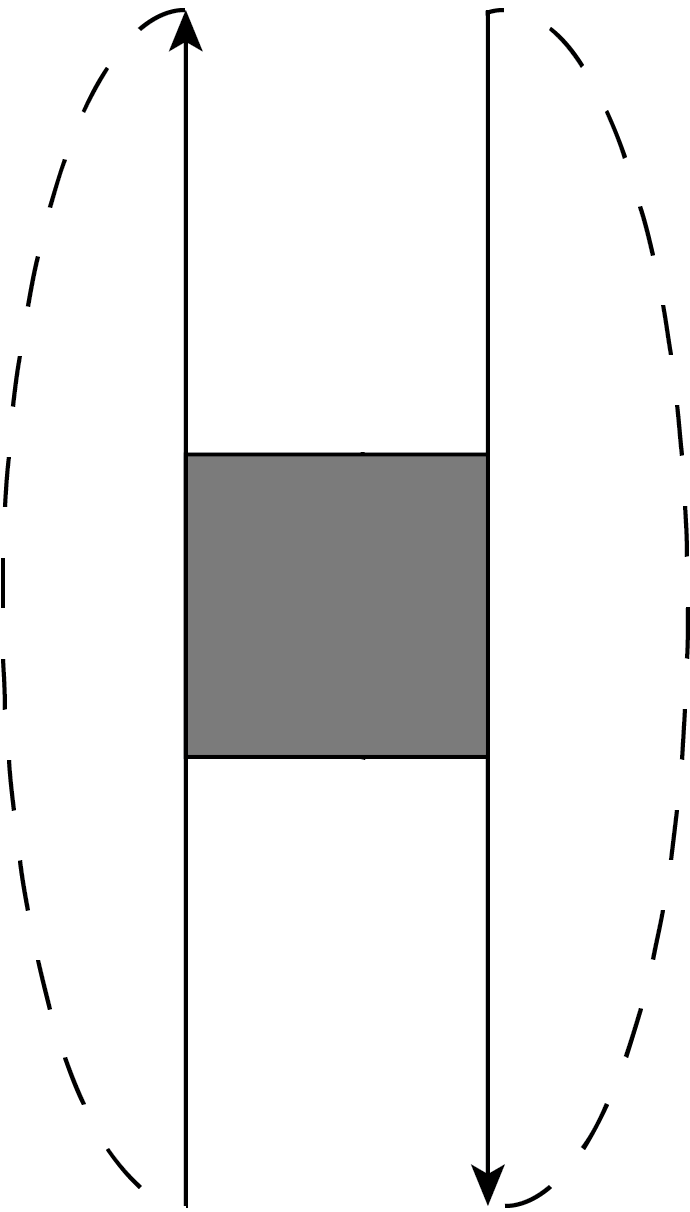}
    \put(-150,60){self-closure}
    \hspace{.5cm}
\includegraphics{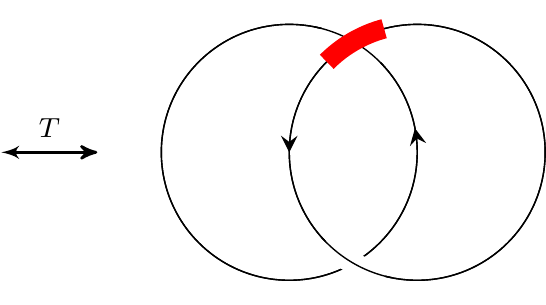}
    \caption{Arc diagram of RNA folding (left), an arc diagram of RNA folding with gray stripe (middle), and a stuck link diagram (right). }
    \label{Closure}
\end{figure}

Using the transformation, $T$, and taking the self-closure of a given arc diagram of an RNA folding, we obtain an associated stuck link diagram. Furthermore, we can compute the stuquandle counting invariant of the RNA folding by computing the stuquandle counting invariant using the associated stuck link diagram.

\begin{example}
Let $X=\mathbb{Z}_{6}$ be the stuquandle with operation defined by $x \ast y = 5x+2y$ and maps $R_1(x,y) = 4x+3y$, $R_2(x,y)=3x+4y$, $R_3(x,y)=x$, and $R_4(x,y)=2x+5y$. By Example~\ref{generalizedaffinestuquandle} with $a=5$, $b=4$ and $e=0$ we obtain that $(X,*,R_1,R_2,R_3,R_4)$ is an oriented stuquandle. Consider the following arc diagram of an RNA folding.

\[
\includegraphics[scale=.5]{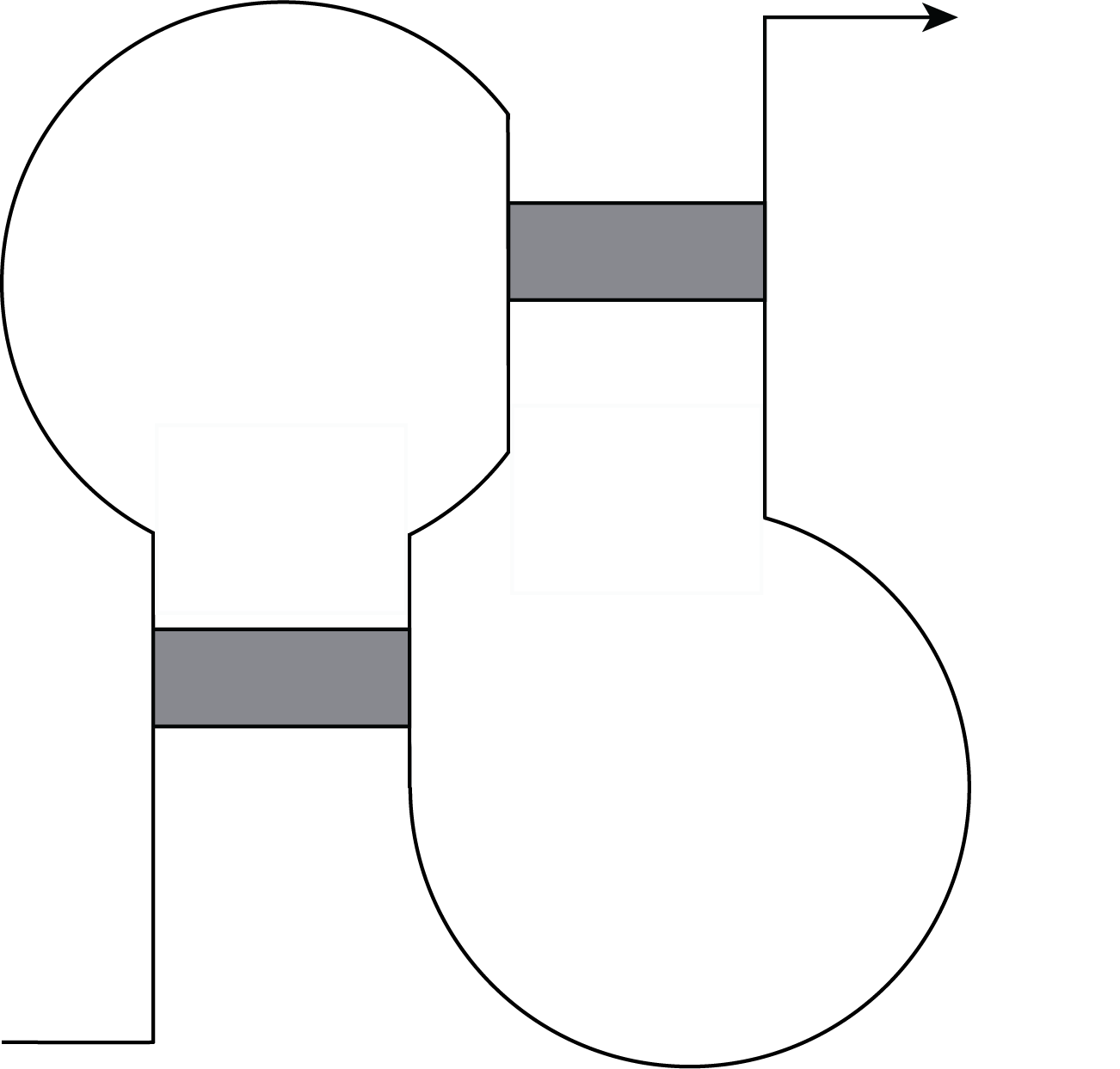}
\hspace{1cm}
\includegraphics{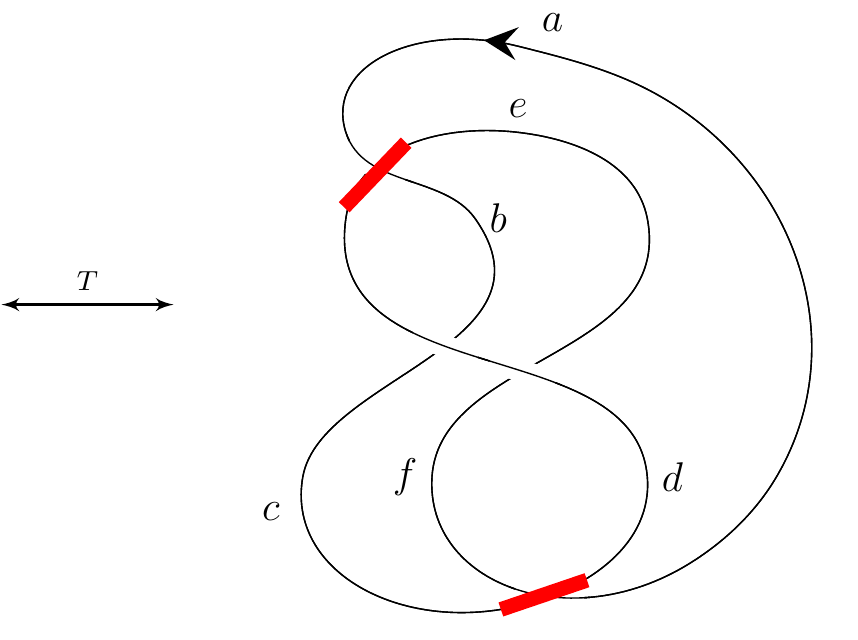}
\]
Using the corresponding stuck knot diagram, we obtain the following coloring equations,
\begin{eqnarray*}
a &=& R_3(f,c),\\
b &=& R_3(a,d),\\
c &=& b \ast d,\\
d &=& R_4(f,c),\\
e &=& R_4(a,d),\\
f &=& e \ast d.
\end{eqnarray*}
 One then gets $a=b=f$, $c=2d-a$, $e=2a-d$ and $3(d-a)=0$.  The equation $3(d-a)=0$ (modulo $6$) implies that $d=a$, or $d=a+2$ or $d=a+4$. Thus, we obtain the \emph{three} types of solutions:
\begin{enumerate}
    \item $a=b=c=d=e=f$;
    \item $a=b=f,c=a+4, d=a+2, e=a-2$;
    \item $a=b=f, c=a+2, d=a+4, e=a-4$.
\end{enumerate}
Therefore, this RNA folding arc has $18$ solutions using the above stuquandle.
\end{example}

Although the stuquandle counting invariant is an effective invariant of RNA foldings, it is not a complete invariant. There are distinct RNA foldings with the same stuquandle counting invariant values. In the next section, we will enhance the stuquandle counting invariant to obtain additional information about an RNA folding to define a more robust invariant. See Section~\ref{CompArcDiagram} for examples of RNA foldings with the same stuquandle counting invariant but distinguished by the enhanced invariant.


\section{Boltzmann Weights and State Sum Invariant}\label{Weight}
 In this section, we will extract additional information about the stuck link, $L$, from $\textup{Hom}(\mathcal{STQ}(L),X)$. In the case that $L$ is a classical link with diagram $D$ and $(X, \ast)$ be a quandle. One very successful idea is to associate at each crossing in $D$ the value of a function $\phi: X \times X \rightarrow A$ (called a \emph{Boltzmann weight}), where $A$ is an abelian group. Next, for each coloring, $f\in \textup{Hom}(\mathcal{Q}(L),X)$, of a link by the quandle, $X$, we take the product in the group $A$ of the values of $\phi$ over the crossings in $D$ colored by $X$. This process gives us what is known as the \emph{cocycle invariant} of the link $L$. To obtain an invariant of knots and links, the function $\phi$ must be chosen such that the Boltzmann weight of $f$ remains unchanged by Reidemeister moves (in other words, $\phi$ satisfies the $2$-cocycle condition).
Several variations of this idea have been defined; see \cites{CJKLS, CES1, AG, EN}. Since we are interested in stuquandles, we will consider the quandle cohomology defined in \cite{CJKLS}. Given a quandle $(X,\ast)$ and abelian group $A$, a function $\phi: X \times X \rightarrow A$ is a \emph{quandle 2-cocycle} if for all $x,y,z \in X$, $\phi$ satisfies the following two conditions
\begin{equation}
 \phi(x,y) + \phi( x \ast y, z) = \phi(x,z) + \phi(x \ast z, y \ast z), \label{cocycle1}
 \end{equation}
 and
 \begin{equation}
 \phi(x,x)=0. \label{cocycle2}
 \end{equation}
Next, we will extend the quandle 2-cocycle to the case of stuquandles.
\begin{definition}\label{cocycle-eqs}
Let $(X,*,R_1,R_2,R_3,R_4)$ be a stuquandle, and $A$ be an abelian group. A Boltzmann weight of $X$ is a triple of maps $\phi, \phi_1, \phi_2: X \times X \rightarrow A$ satisfying the following conditions.
\begin{enumerate}
    \item For all $x \in X$, $\phi(x,x)=0$.
    \item For all $x,y \in X$,
    \begin{align*}
        &\text{(i)}&\phi_2(y,x) +\phi(R_3(y,x),R_4(y,x) &= \phi(x,y)+\phi_2(x * y,y),\\
        &\text{(ii)}&\phi_1(x,y)+\phi(R_1(x,y),R_2(x,y)) &= \phi(x,y) + \phi_1(y, x* y).
    \end{align*}
    \item For all $x,y,z \in X$,
    \begin{align*}
        &\text{(i)}&\phi(x,y)+\phi(x*y,z) &= \phi(x,z)+\phi(x*z,y*z),\\
        &\text{(ii)}&\phi_2(y, z\,\bar{*}\,x)+\phi(R_3(y,z\,\bar{*}\, x),x)-\phi(z\,\bar{*}\, x,x) 
        &= \phi(y,x) + \phi_2(y * x, z) - \phi(R_4(y* x,z)\,\bar{*}\, x,x),\\
        &\text{(iii)}&\phi_1(z\,\bar{*}\,x,y)+\phi(R_1(z\,\bar{*}\,x,y),x)-\phi(z \,\bar{*}\,x,x)
        &= \phi(y,x) + \phi_1(z,y* x)-\phi(R_2(z,y* x)\,\bar{*}\,x,x),\\
        &\text{(iv)}&\phi(x\,\bar{*}\,R_3(y,z),z) -\phi(x\,\bar{*}\,R_3(y,z), R_3(y,z))
        &= \phi(x, R_4(y,z)) -\phi((x* R_4(y,z)) \,\bar{*}\, y,y),\\
        &\text{(v)}&\phi(x\,\bar{*}\, R_1(z,y),z) -\phi(x\,\bar{*}\, R_1(z,y), R_1(z,y)) 
        &= \phi(x, R_2(z,y))-\phi((x* R_2(z,y)) \,\bar{*} \,y,y).
    \end{align*}
\end{enumerate}
\end{definition}

\begin{remark}
Note that $\phi$ is a quandle 2-cocycle.
\end{remark}

The definition of a Boltzmann weight of a stuquandle is motivated by the oriented stuck Reidemeister moves for stuck links using the contribution rule at classical crossings and stuck crossings in Figure~\ref{crossingrule}.

\begin{figure}[h]
\centering
    \includegraphics{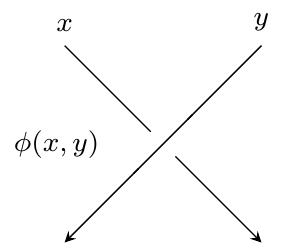}
\hspace{1cm}
    \includegraphics{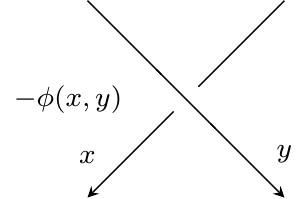}

    \includegraphics{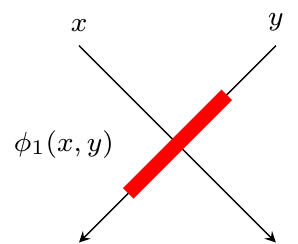}
\hspace{2cm}
    \includegraphics{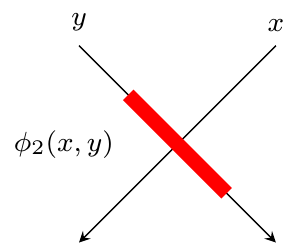}
\caption{Contribution rules for weights at stuck crossings.}
\label{crossingrule}
\end{figure}

\[
\includegraphics{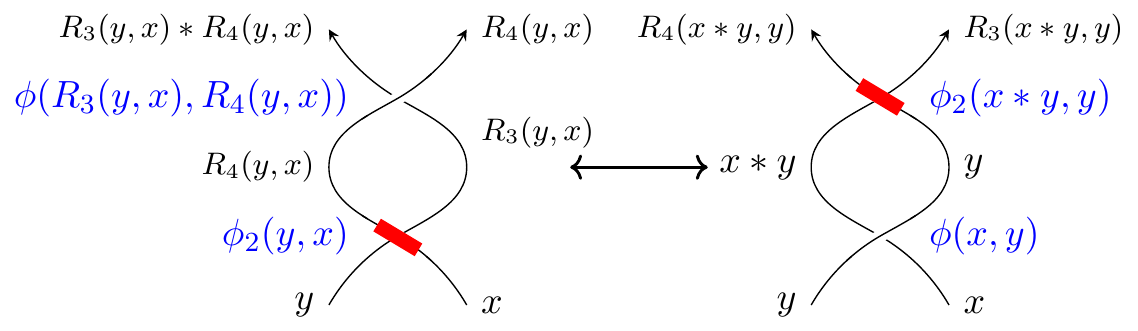}
\]
\[
\includegraphics{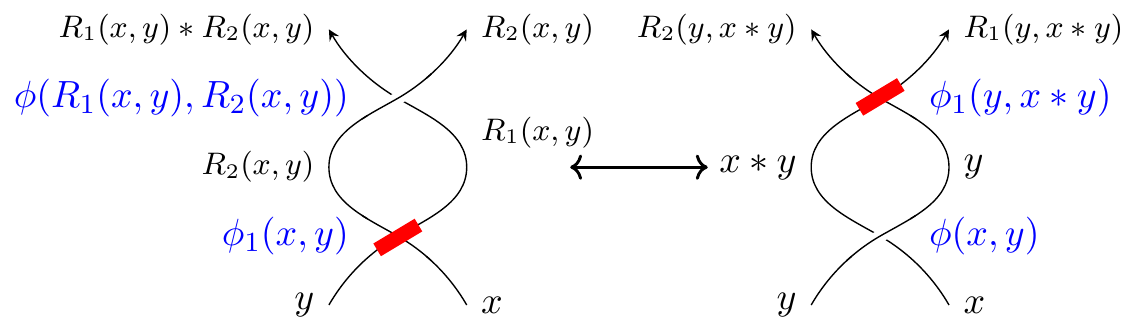}
\]
\[
\includegraphics{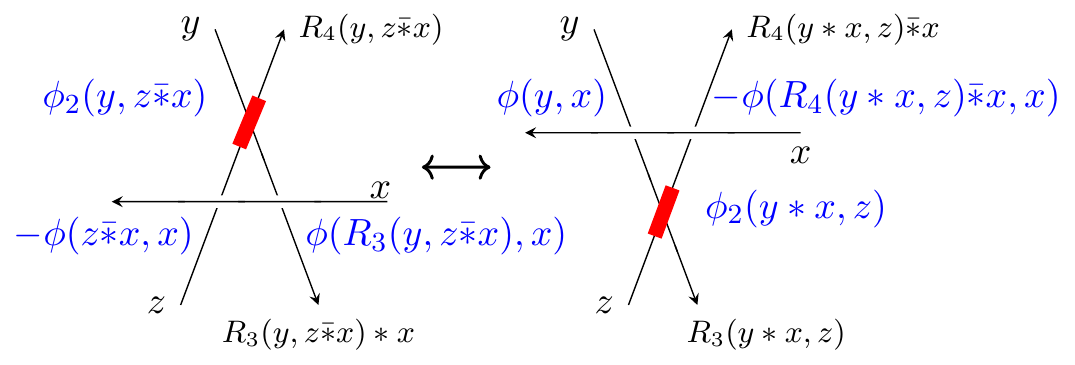}
\]
\[
\includegraphics{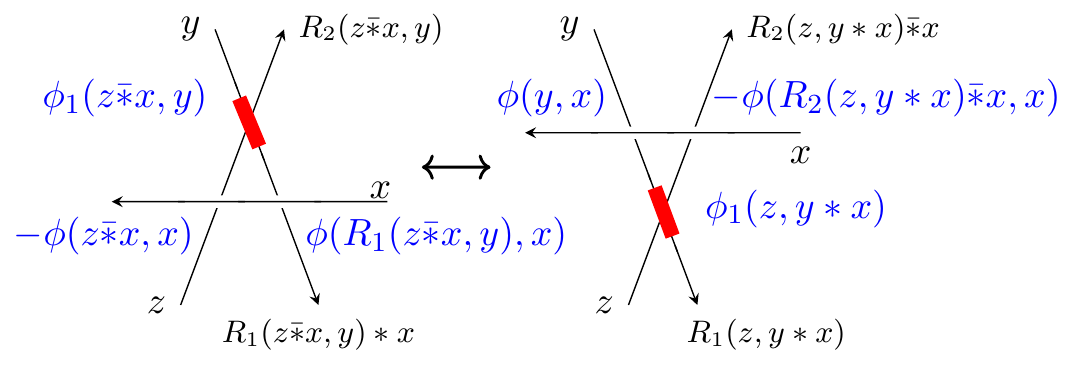}
\]
\vspace*{-8mm}
\[
\includegraphics{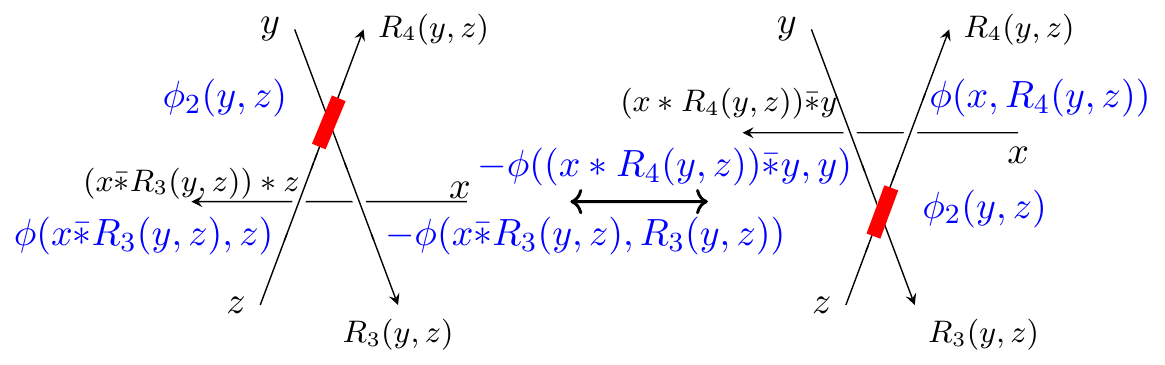}
\]
\[
\includegraphics{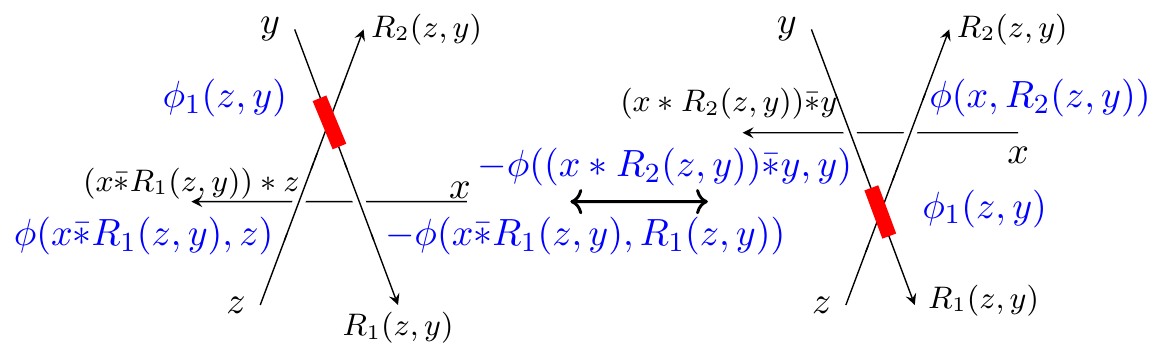}
\]

\begin{definition}\label{compatabilityrules}
Let $(X,*,R_1,R_2,R_3,R_4)$ be a stuquandle, $A$ an abelian group, and $(\phi, \phi_1, \phi_2)$ a Boltzmann weight with coefficients in $A$. We say that $\phi$ and $\phi_1$ are strongly compatible if for all $x,y,z \in X$, 
\begin{eqnarray}
    \phi_1(x,y) &=&\phi_1(y, x * y),\\
    \phi_1(z\,\bar{*}\, x,y) &=& \phi_1(z,y * x).
\end{eqnarray}
Similarly, we say that $\phi$ and $\phi_2$ are strongly compatible if for all $x,y,z \in X$
\begin{eqnarray}
    \phi_2(y,x) &=&\phi_2(x * y,y),\\
    \phi_2(y, z\,\bar{*}\, x) &=& \phi_2(y * x, z).
\end{eqnarray}

\end{definition}

\begin{definition} \label{compatibleweights} 
Let $(X,*,R_1,R_2,R_3,R_4)$ be a stuquandle, $A$ an abelian group, and $(\phi, \phi_1, \phi_2)$ a Boltzmann weight with coefficients in $A$. Let $L$ be an oriented stuck knot or link with diagram $D$. 
\begin{enumerate}
    \item For each $f \in \textup{Hom}(\mathcal{STQ}(L),X)$, we define the \emph{Boltzmann weight} of $f$, denoted by $BW(f)$, to be the sum over all the crossings in $D$ subject to the contribution rules in Figure~\ref{crossingrule}.
    \item We can define a \emph{single-variable state sum invariant} by  $$\Phi_X^{\phi,\phi_1,\phi_2}(L) = \sum_{f \in \textup{Hom}(\mathcal{STQ}(L),X) } u^{BW(f)}.$$
    \item If $\phi$ and $\phi_1$ are compatible, we define the \emph{two-variable state sum invariant} by $$\Phi_X^{\phi,\phi_1,\phi_2}(L) = \sum_{f \in \textup{Hom}(\mathcal{STQ}(L),X) } u^{BW(f)} v^{BW_{\phi_{1}}(f)}.$$
Similarly, if $\phi$ and $\phi_2$ are compatible, we define the \emph{two-variable state sum invariant} by $$\Phi_X^{\phi,\phi_1,\phi_2}(L) = \sum_{f \in \textup{Hom}(\mathcal{STQ}(L),X) } u^{BW(f)} w^{BW_{\phi_{2}}(f)}.$$
Note that the \emph{partial Boltzmann weights} denoted by $BW_{\phi_1}$ and $BW_{\phi_2}$ are the sums of $\phi_1$ contributions and $\phi_2$ contributions, respectively.
\item Lastly, if both $\phi_1$ and $\phi_2$ are compatible with $\phi$, we define the \emph{three-variable state sum invariant} by $$\Phi_X^{\phi,\phi_1,\phi_2}(L) = \sum_{f \in \textup{Hom}(\mathcal{STQ}(L),X) } u^{BW_{\phi}(f)} v^{BW_{\phi_{1}}(f)}w^{BW_{\phi_{2}}(f)}.$$
\end{enumerate}
\end{definition}
\noindent By construction, we have the following result.
\begin{theorem}
Let $X$ be a stuquandle, $A$ be an abelian group, and let $\phi, \phi_1$ and $\phi_2$ be respectively the Boltzmann weights at regular positive crossings, positive stuck crossings, and negative stuck crossings with coefficients in $A$. Then $\Phi_X^{\phi, \phi_1, \phi_2}$ is an invariant of stuck knots and links.
\end{theorem}
\begin{proof}
As in the case of classical and singular knot theories, there is a one-to-one correspondence between colorings before and after each of the stuck Reidemeister moves in the generating set in Figure \ref{stuckrmoves}.  The invariance follows directly from the equations involving $\phi$, $\phi_1$, and $\phi_2$ in Definition~\ref{cocycle-eqs}. 
\end{proof}


\section{Computations for Stuck Knots and Links}\label{CompStuck}

This section collects a few examples and computations of the state sum invariant. The examples in this section were computed using our custom \texttt{Python} code.
\begin{example}
Let $X = \Z_{4}$ with the operations $*,\bar{*}, R_1, R_2, R_3, R_4: X\times X \to X$ and $\phi, \phi_1, \phi_2: X \times X \to \Z_{4}$ defined below.
\begin{align*}
    x*y &= 3x+2y &\phi\,(x,y) &= 2xy+2y\\
    x\,\bar{*}\,y &=3x+2y & \phi_1(x,y) &= 2x+2y\\
    R_1(x,y) & = x+2y^2 &     \phi_2(x,y) & = 2x\\
    R_2(x,y) &= 2x^2+y \\
    R_3(x,y) &= 3x \\
    R_4(x,y) &= 2x+y
\end{align*}
\begin{figure}[h]
    \begin{minipage}{0.48\textwidth}
    \centering
\includegraphics{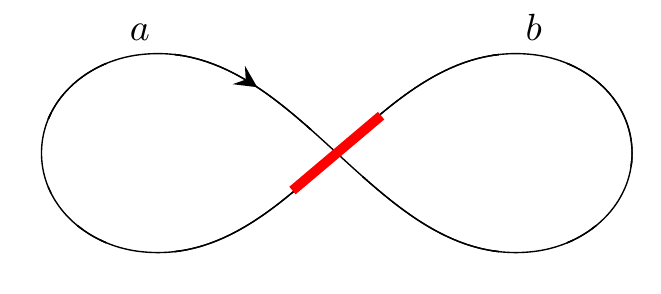}
\caption{ The stuck knot infinity $0^{k+}_1$}
\label{stuckinfinity}
    \end{minipage}
    \begin{minipage}{0.48\textwidth}
    \centering
\includegraphics{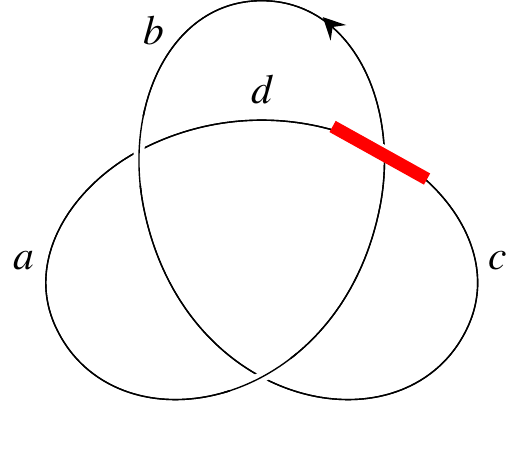}
    \caption{ The stuck trefoil $2^{k-}_1$}
    \label{stucktrefoil}
    \end{minipage}
\end{figure}
Given the stuquandle defined above, the colorings for Figure \ref{stuckinfinity} are $(0,0),(0,2),(2,0),(2,2)$, and the colorings for Figure \ref{stucktrefoil} are $(0, 0, 0, 0)$, $(1, 3, 3, 1), (2, 2, 2, 2), (3, 1, 1, 3)$. Thus, the coloring invariant for both knots is $4$. However, if we consider the state sum invariant with Boltzmann weights $\phi, \phi_1, \phi_2$, then
\[\Phi_X^{\phi,\phi_1, \phi_2}(0^{k+}_1) = 4\]
and
\[\Phi_X^{\phi,\phi_1, \phi_2}(2^{k-}_1) = 2u^2+2.\]
Therefore, the state sum invariant can distinguish the knots.
\end{example}

Using our custom \texttt{Python} code, we computed Boltzmann weights for specific stuck knots and links based on singular knots and links known as two-bouquet graphs of type $K$ in \cite{O} with both a positive stuck crossing($L^{k+}$) and negative stuck crossing($L^{k-}$) variation.\par

\begin{example} Let the stuquandle $X = \Z_4$ with operations $*,\bar{*}, R_1, R_2, R_3, R_4: X\times X \to X$ and $\phi, \phi_1, \phi_2: X \times X \to A=\Z_{4}$ defined below.
\begin{align*}
    x*y &= x    &  \phi\,(x,y) &= y^2+xy+2x\\
    x\,\bar{*}\,y &= x  &     \phi_1(x,y) &= x^2+2y\\
    R_1(x,y) & = 2x+3y  &     \phi_2(x,y) & = 2x+2y \\
    R_2(x,y) &= 3x+2y \\
    R_3(x,y) &=  y\\
    R_4(x,y) &= x
\end{align*}
The results are collected in the table:\\

\centering
\setlength{\tabcolsep}{8pt} 
\renewcommand{\arraystretch}{1.25} 
\begin{tabular}{c|c|c}
    $Col_X(L)$ & $\Phi_X^{\phi,\phi_1,\phi_2}(L)$  & $L$\\
    \hline
    $4$ & $4$ & $1^{l-}_1$, $3^{l-}_1$, $4^{l-}_1$, $5^{l-}_1$ \\
     & $2u^3+2$ & $0^{k+}_1$,$2^{k+}_1$, $3^{k+}_1$, $4^{k+}_1$ \\
    \hline
    $8$ 
     & $2u^3+2u+4$ & $1^{l+}_1$, $3^{l+}_1$, $4^{l+}_1$, $5^{l+}_1$ \\
    \hline
    $16$ & $16$ & $4^{k-}_1$\\
     & $8u^2+8$ & $0^{k-}_1$\\
     & $8u^3+8$ & $2^{k-}_1$, $3^{k-}_1$.
\end{tabular}
\end{example}

Using the compatibility rules from Definition \ref{compatabilityrules} and the state sum invariant defined in Definition \ref{compatibleweights}, we can distinguish most of oriented stuck knots derived from the oriented left-handed trefoil. 

\begin{example} Let the stuquandle $X = \Z_8$ with operations $*,\bar{*}, R_1, R_2, R_3, R_4: X\times X \to X$ and $\phi: X \times X \to \Z_{4}, \phi_1: X \times X \to \Z_{4}, \phi_2: X \times X \to \Z_{4}$ defined below.
\begin{align*}
    x*y &= 5x+4y    &  \phi\,(x,y) &= 2x+2y\\
    x\,\bar{*}\,y &= 5x+4y  &     \phi_1(x,y) &= 3x+3y\\
    R_1(x,y) & = y  &     \phi_2(x,y) & = xy+x+y \\
    R_2(x,y) &= 5x+4y \\
    R_3(x,y) &=  x\\
    R_4(x,y) &= 4x+5y
\end{align*}
Notice that both $\phi_1$ and $\phi_2$ are compatible with $\phi$. The results are collected in the table:\\

\centering
\begin{tabular}{c|c|l}
    $Col_X(L)$ & $\Phi_X^{\phi,\phi_1,\phi_2}(L)$  & \hspace{.5cm}$L$\\
    \hline
     & $4$ & \parbox[c]{1em} {\resizebox {3cm} {!} {
\includegraphics{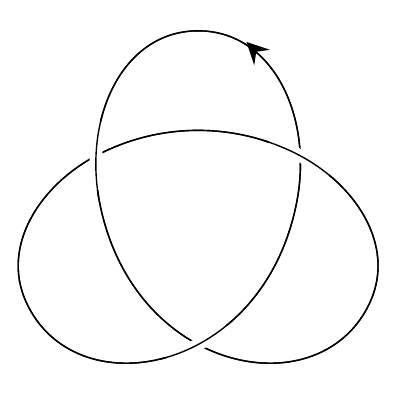}
\includegraphics{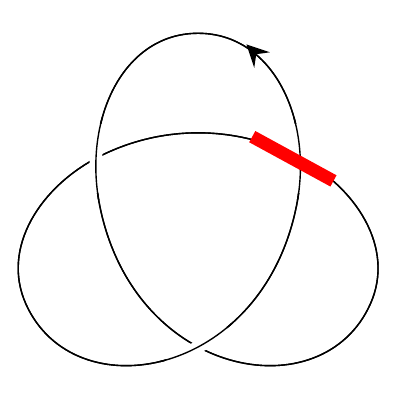}
    }}\\
     & $2v+2$ &  \parbox[c]{1em} {\resizebox {1.5cm} {!} {
\includegraphics{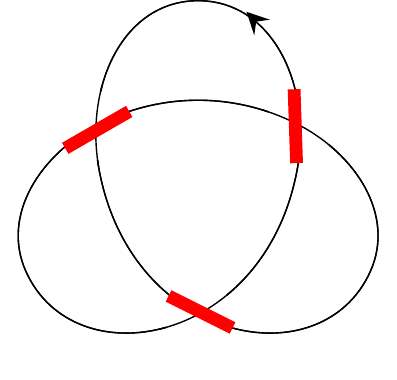}
    }}  \\
     & $2v^2+2$ & \parbox[c]{1em} {\resizebox {3cm} {!} {
\includegraphics{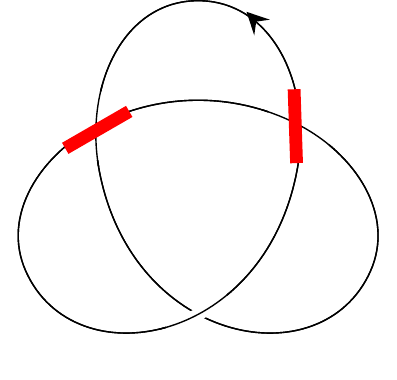}
\includegraphics{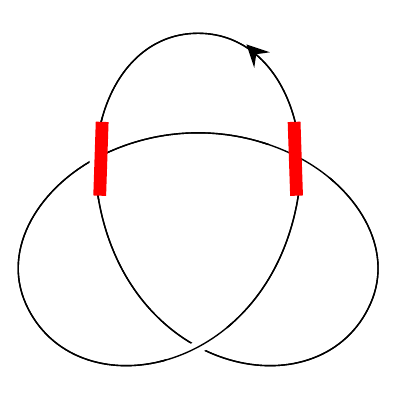}
    }} \\ 
     & $2v^3+2$ &  \parbox[c]{1em}{\resizebox {1.5cm} {!} {
\includegraphics{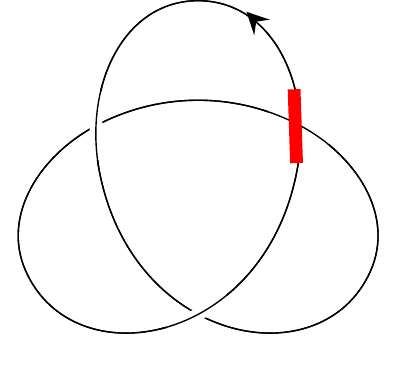}
     }}\\
    $4$ & $2v^2w^3+2$ &  \parbox[c]{1em} {\resizebox {1.5cm} {!} {
\includegraphics{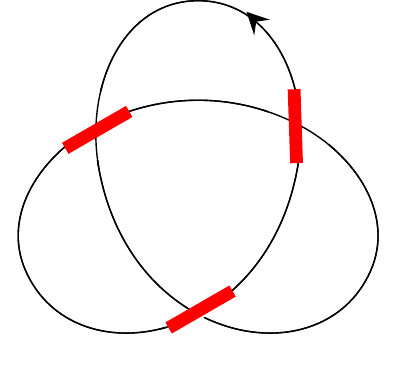}
    }} \\
     & $2v^3w^2+2$ &  \parbox[c]{1em} {\resizebox {1.5cm} {!} {
\includegraphics{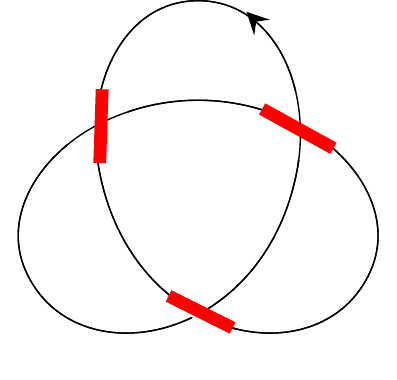}
    }} \\
     & $2v^3w^3 + 2$ &  \parbox[c]{1em}{\resizebox {1.5cm} {!} {     
\includegraphics{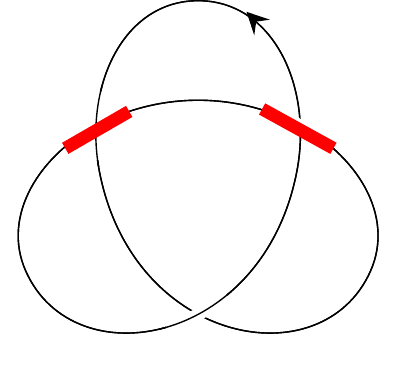}
    }} \\
     & $2w+2$ &  \parbox[c]{1em} {\resizebox {1.5cm} {!} {
\includegraphics{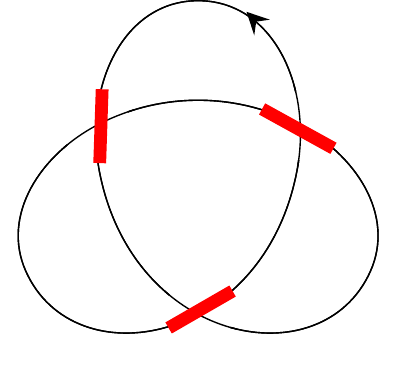}
    }}  \\
     & $2w^2+2$ &  \parbox[c]{1em}{\resizebox {1.5cm} {!} {
\includegraphics{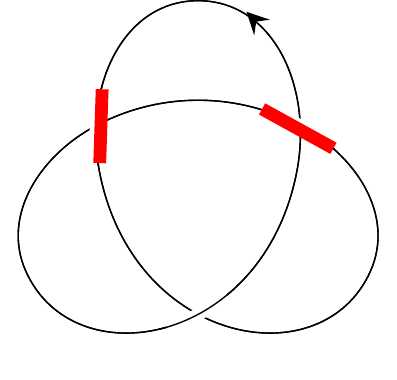}
    }}\\
\end{tabular}
\end{example}

\section{Computations for RNA Foldings}\label{CompArcDiagram} 
In this section, we compute the counting and state sum invariants of arc diagrams of RNA foldings under the transformation, $T$, defined in Section \ref{arcdiagrams}. In the following example, we use the stuck knots in Figure \ref{arcdiagramknots} derived from the arc diagrams \textbf{i}-\textbf{iv} and $\mathbf{1}$-$\mathbf{10}$ in \cite{KM} on pages 347 and 353 respectively. 
The examples in this section were computed using our custom \texttt{Python} code.

\begin{figure}[h!]
\includegraphics{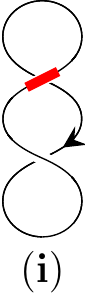}
\includegraphics{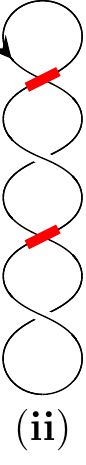}
\includegraphics{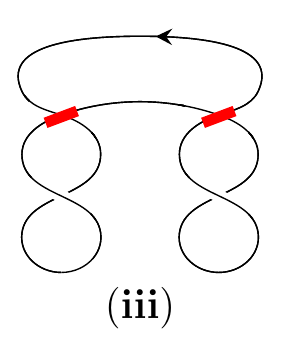}
\includegraphics{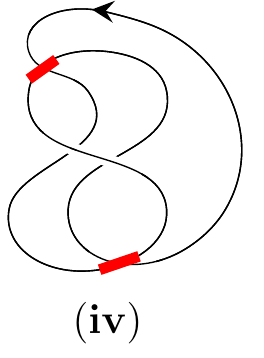}
\includegraphics{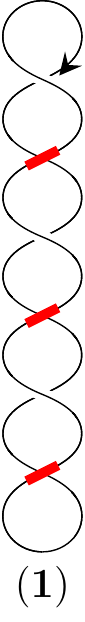}
\includegraphics{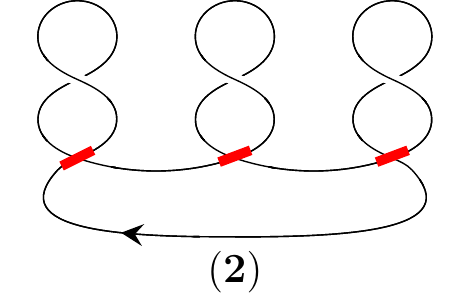}
\includegraphics{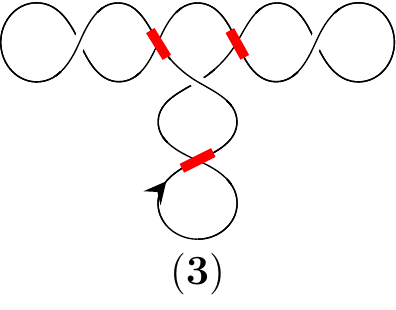}
\includegraphics{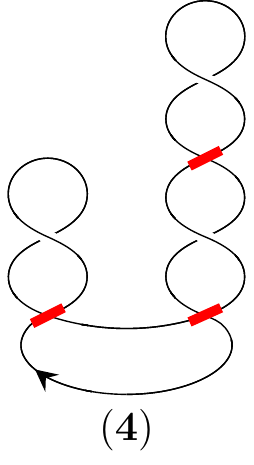}
\includegraphics{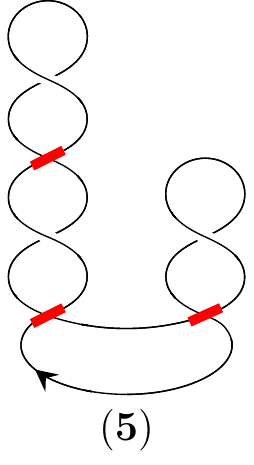}
\includegraphics{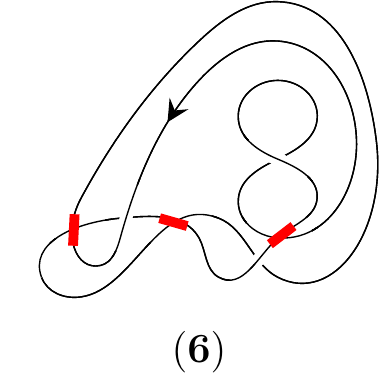}
\includegraphics{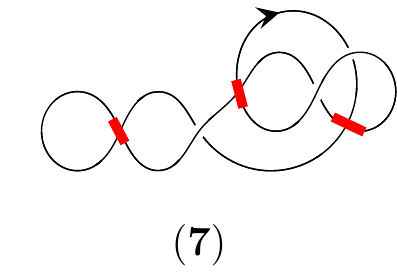}
\includegraphics{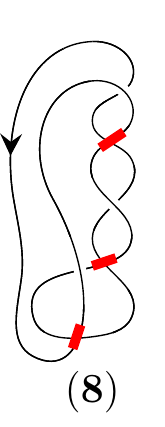}
\includegraphics{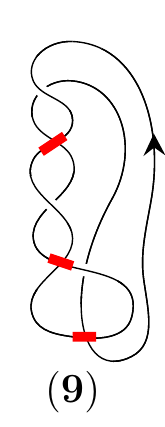}
\includegraphics{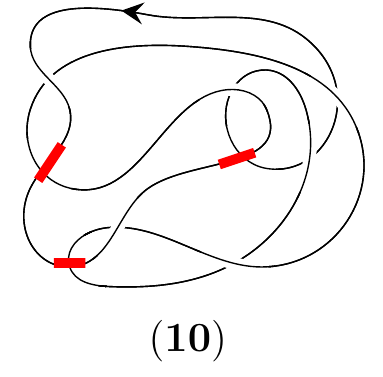}
\caption{Arc Diagrams of RNA foldings}
\label{arcdiagramknots}
\end{figure}

\begin{example}
Let $X_1 = \Z_3$ with operations $*,\bar{*}, R_1, R_2, R_3, R_4:X_1\times X_1 \to X_1$ and $\phi,\phi_1: X_1 \times X_1 \to \Z_3$, $\phi_2: X_1 \times X_1 \to \Z_3$ defined below.
\begin{align*}
    x*y &= x    &  \phi\,(x,y) &= x+2y\\
    x\,\bar{*}\,y &= x  &     \phi_1(x,y) &= y\\
    R_1(x,y) & = y  &     \phi_2(x,y) & = xy+x+y \\
    R_2(x,y) &= x \\
    R_3(x,y) &=  y\\
    R_4(x,y) &= x
\end{align*}
\noindent Let $X_2 =  \Z_4$ with operations $*',\bar{*}', R'_1, R'_2, R'_3, R'_4:X_2\times X_2 \to X_2$ and $\phi',\phi'_1: X_2 \times X_2 \to \Z_4$, $\phi'_2: X_2 \times X_2 \to \Z_4$ defined below.
\begin{align*}
    x*'y &= x    &  \phi'\,(x,y) &= x+3y\\
    x\,\bar{*}'\,y &= x  &     \phi'_1(x,y) &= 2y\\
    R'_1(x,y) & = y  &     \phi'_2(x,y) & = xy+x+y \\
    R'_2(x,y) &= x \\
    R'_3(x,y) &=  y\\
    R'_4(x,y) &= x
\end{align*}
The results are collected in the table:\\

\centering
\setlength{\tabcolsep}{8pt}
\renewcommand{\arraystretch}{1.5} 
\begin{tabular}{c|c|c}
    $\left(Col_{X_1}(L), Col_{X_2}(L)\right)$ & $\left(\Phi_{X_1}^{\phi,\phi_1,\phi_2}(L), \Phi_{X_2}^{\phi',\phi'_1,\phi'_2}(L)\right)$  & $L$\\
    \hline
    $(3, 4)$ & $(w+2, 2w^2+2)$ & \emph{(iv)} \\
    \hline
    $(9, 16)$ & $(5w^2+2w+2, 8w^3 + 2w^2 + 4w + 2)$ & \emph{(i)}\\
     & $(9, 4w^3+2w^2+8w+2)$ & \emph{(10)}\\
     & $(2w^2+2w+5, 4w^2+8w+4)$ & \emph{(8)}, \emph{(9)} \\
     & $(2w^2+2w+5, 4w^3+2w^2+8w+2)$ & \emph{(6)}, \emph{(7)} \\
    \hline
    $(27,64)$ & $(6w^2+15w+6, 8w^3+32w^2+8w+16)$ & \emph{(ii)}, \emph{(iii)}\\
    \hline
    $(81, 256)$ & $(18w^2+18w+45, 64w^3+32w^2+128w+32)$ & \emph{(2)}, \emph{(3)} \\
     & $(24w^2+24w+33, 72w^3+48w^2+88w+48)$ & \emph{(1)}, \emph{(4)}, \emph{(5)}.
\end{tabular}
\end{example}

\section*{Acknowledgement} 
Mohamed Elhamdadi was partially supported by Simons Foundation collaboration grant 712462.


\bibliography{Refs}
\bibliographystyle{plain}

\end{document}